\documentclass[sn-mathphys,Numbered]{sn-jnl}

\usepackage{graphicx}%
\usepackage{multirow}%
\usepackage{amsmath,amssymb,amsfonts}%
\usepackage{amsthm}%
\usepackage{mathrsfs}%
\usepackage[title]{appendix}%
\usepackage{xcolor}%
\usepackage{textcomp}%
\usepackage{manyfoot}%
\usepackage{booktabs}%
\usepackage{algorithm}%
\usepackage{algorithmicx}%
\usepackage{algpseudocode}%
\usepackage{listings}%

\theoremstyle{thmstyleone}%
\newtheorem{theorem}{Theorem}
\newtheorem{proposition}[theorem]{Proposition}%

\theoremstyle{thmstyletwo}%

\theoremstyle{thmstylethree}%
\newtheorem{definition}{Definition}%

\begin{document}

\title[D.Q]{DEFORMATION QUANTIZATION OF A HESSIAN KV-STRUCTURE ON $\mathrm{I\! R}^{2}$.}


\author*[1]{\fnm{} \sur{ Herguey Mopeng}}\email{mopengpeter@gmail.com}
\author[2]{\fnm{} \sur{Prosper Rosaire Mama  Assandje }}\email{mamarosaire@facsciences-uy1.cm}\equalcont{These authors
contributed equally to this work.}
\author[3]{\fnm{} \sur{Joseph Dongho }}\email{josephdongho@yahoo.fr}
\equalcont{These authors contributed equally to this work.}
\author[1]{\fnm{} \sur{Armand Tsimi  }}\email{tsimije@yahoo.fr}
\equalcont{These authors contributed equally to this work.}

\affil*[1,4]{\orgdiv{Department of Mathematics}, \orgname{University
of Douala}, \orgaddress{\street{rectorat@univ-douala.com},
\city{Douala}, \postcode{2701}, \state{Littoral},
\country{Cameroon}}}

  \affil[2]{\orgdiv{Department of Mathematics}, \orgname{University of Yaounde 1},
\orgaddress{\street{usrectorat@.univ-yaounde1.cm}, \city{yaounde},
\postcode{337}, \state{Center}, \country{Cameroon}}}

\affil[3]{\orgdiv{Department of Mathematics and Computer Science},
\orgname{University of Maroua},
\orgaddress{\street{decanat@fs.univ-maroua.cm}, \city{Maroua},
\postcode{814}, \state{Far -North}, \country{Cameroon}}}

\affil[4]{\orgdiv{Department of Mathematics }, \orgname{University
of Pala}, \orgaddress{\street{mopengpeter@gmail.com}, \city{Pala},
\postcode{20}, \state{Far -North}, \country{Tchad}}}

\abstract{This paper studies the quantization of the deformation of
Hessian structures on a two-dimensional vector space, in the
framework of Koszul-Vinberg algebras.  We analyze how Hessian
structures can be deformed to obtain quantum structures while
preserving certain geometric and algebraic properties. The aim of
the paper is to establish links between deformation theory and
Hessian geometry.}




\maketitle

\section{Introduction}\label{sec1}
Deformation quantization has emerged as a fundamental tool in the
development of quantum theories, allowing classical algebraic
structures to be linked to quantum formulations. In this context,
the study of Hessian structures of Koszul-Vinberg algebras offers a
rich and promising perspective. These algebras, which are
characterized by special geometric and algebraic properties, play a
crucial role in the understanding of dynamical systems and field
theories. Koszul algebras, which include Lie and Poisson algebras,
play a fundamental role in mathematics, especially in algebraic
geometry and representation theory. These algebraic structures are
characterized by duality and solvability properties that facilitate
the study of modules and associated representations. However, since
every deformation theory generates its own cohomological theory,
it's important to note that this is not the case for all deformation
theories. The aim of this paper is to explore the deformation
quantization of a Hessian structure associated with a Koszul-Vinberg
algebra on $\mathrm{I\!R}^{2}$, where the main question is: how can
a Hessian KV structure on $\mathrm{I\!R}^{2}$ be quantified by
deformation? In \cite{Gerstenhaber}, several authors have addressed
the question. Andr$\acute{e}$ Haefliger \cite{HAEFLIGER} had
addressed the Cohomology of Lie algebras associated with vector
fields. He explored the relationships between the geometry of
variants and the algebraic structure of vector fields, highlighting
the topological and geometric implications of this cohomology. In
\cite{Pichereau}, had addressed the cohomology of Poisson structures
in dimension three, focusing on the algeometric and geometric
proprietorships of Poisson variants. She had explored the
relationships between Poisson cohomology and other types of
cohomology, as well as the implications of these relationships for
the classification and study of Poisson structure diformities. It
was not until the recent work of M.N. Boyom
\cite{BOYOM1,BOYOM2,BOYOM3,BOYOM5}, to see that the theory of
KV-algebraic formation has its own cohomological theory:
KV-cohomology. This tool makes it possible to describe the
deformation of KV-algebras.  Ferdinand Ngakeu et al.\cite{Boyom9}
had examined KV-cohomology, which is a form of cohomology associated
with specific geometric structures, particularly in the context of
flat refinement variants. The authors introduce the fundamental
concepts of differential geometry and explore how these notions
apply to variants with a flat affine structure. They discuss the
topological and geometric properties of these variants, as well as
applications of KV cohomology to the study of their structure.
Hessian structures, on the other hand, are geometric objects that
appear in the analysis of variants and modular spaces. They provide
a framework for studying the local and global properties of
dynamical systems and differential equations. The relationship
between these Hessian structures and Koszul-Vinberg algebra allows
us to explore deep interactions between algebra and geometry. In
\cite{Barbaresco1}, Fr$\acute{e}$d$\acute{e}$ric Barbaresco explores
the relationships between information geometry, Lie group
thermodynamics and the concepts of geometric temperature and
capacity. Quantization aims to establish a correspondence between
classical objects and their quantum analogues while preserving
certain structural properties.  In \cite{Dupont}, Dupont
Gr$\acute{e}$goire explores how the properties of algebraic
structures can be modified in a continuous and systematic way by
introducing mathematical tools such as cohomology and deformation
classes. The ability to form a Hessian structure while preserving
its essential properties is crucial for understanding how these
structures can be adapted to quantum contexts. Several authors have
worked on the question of how algebraic structures that transform
under the effect of deformation can provide tools for modeling
complex physical phenomena. Following the example of Gerstenhaber
\cite{Gerstenhaber}, who studied the fundamental concepts of ring
and algebra formation. Important results on the Cohomology of
algebras and their relation to deformations were presented. Amnon
Yekutieli dealt with the quantification of deformations   in the
context of algebraic geometry.  In \cite{Butin}, F. Butin  had
studied Poisson structures on polynomial algebras, focusing on
cohomology  associates and deformations of these structures.
F.Bayen, explored  the theory of deformations in the context of
symplectic structures and their connection with quantization,
presenting a systematic approach to understanding how symplectic
structures can be deformed and how these deformations affect
quantization. Guillemin  and Sternberg \cite{gg} have explored the
relationship between dation theory and pseudo-groups. In the light
of all this work, we were able to show that for all $u=(x,y)$,
$v=(x^{'},y^{'})$ et $w=(x^{''},y^{''})$ three vectors
of $\mathrm{I\!R}^{2}$, with $\mu(u,v)=\mu_{1}(u,v)e_{1}+\mu_{2}(u,v)e_{2}$ \\
in a basis $(e_{1},e_{2})$ de $\mathrm{I\!R}^{2}$,and for all
matrices of the bilinear forms $\mu_{1}$ and $\mu_{2}$, respectively
\begin{center}
     $\Gamma_{1}= \begin{pmatrix}
          \Gamma_{11}^{1}&\Gamma_{12}^{1}\\
           \Gamma_{21}^{1}&\Gamma_{22}^{1}
            \end{pmatrix},
       \Gamma_{2}= \begin{pmatrix}
          \Gamma_{11}^{2}&\Gamma_{12}^{2}\\
           \Gamma_{21}^{2}&\Gamma_{22}^{2}
            \end{pmatrix}$
\end{center} that $\mu$  is a KV-structure if and only if
\begin{eqnarray*}
\begin{cases}
 \Gamma_{21}^{1}(\Gamma_{12}^{2}-\Gamma_{11}^{1}-\Gamma_{21}^{2})+\Gamma_{12}^{1}(\Gamma_{11}^{1}-\Gamma_{21}^{2})
+\Gamma_{22}^{1}\Gamma_{11}^{2}=0 \\
\Gamma_{22}^{1}(2\Gamma_{12}^{2}-\Gamma_{11}^{1}-\Gamma_{21}^{2})+\Gamma_{12}^{1}(\Gamma_{12}^{1}-\Gamma_{22}^{2})=0\\
\Gamma_{11}^{2}(\Gamma_{12}^{1}-2\Gamma_{21}^{1}+\Gamma_{22}^{2})+\Gamma_{21}^{2}(\Gamma_{11}^{1}-\Gamma_{21}^{2})=0\\
\Gamma_{12}^{2}(\Gamma_{12}^{1}-\Gamma_{21}^{1}+\Gamma_{22}^{2})+\Gamma_{21}^{2}(\Gamma_{12}^{1}-\Gamma_{22}^{2})-\Gamma_{11}^{2}\Gamma_{22}^{1}=0.
\end{cases}
\end{eqnarray*} We show that, on $\mathrm{I\!R}^{2}$, the following KV-structures are
Hessian is given by
 \begin{eqnarray*}
\mu=\biggl(\begin{pmatrix}
         0&\Gamma_{22}^{2}\\
          \Gamma_{22}^{2} &0
            \end{pmatrix},
        \begin{pmatrix}
         \Gamma_{11}^{2}  &0 \\
           0&\Gamma_{22}^{2}
            \end{pmatrix}\biggr);
\mu=\biggl(\begin{pmatrix}
         \Gamma_{11}^{1} &0\\
          0 &\Gamma_{22}^{1}
            \end{pmatrix},
        \begin{pmatrix}
         0  &\Gamma_{11}^{1} \\
           \Gamma_{11}^{1}&\Gamma_{22}^{2}
            \end{pmatrix}\biggr);
\mu=\biggl(\begin{pmatrix}
         \Gamma_{11}^{1} &0\\
          0 &\Gamma_{22}^{1}
            \end{pmatrix},
        \begin{pmatrix}
         \Gamma_{11}^{2}  &0 \\
           0&\Gamma_{22}^{2}
            \end{pmatrix}\biggr);
\end{eqnarray*}
\begin{eqnarray*}
\mu=\biggl(\begin{pmatrix}
         \Gamma_{11}^{1} &\Gamma_{22}^{2}\\
          \Gamma_{22}^{2} &0
            \end{pmatrix},
        \begin{pmatrix}
         \Gamma_{11}^{2}  &0 \\
           0&\Gamma_{22}^{2}
            \end{pmatrix}\biggr);
\mu=\biggl(\begin{pmatrix}
         \Gamma_{11}^{1} &0\\
          0&\Gamma_{22}^{1}
            \end{pmatrix},
        \begin{pmatrix}
         0&\Gamma_{11}^{1} \\
           \Gamma_{11}^{1}&0
            \end{pmatrix}\biggr).
\end{eqnarray*} We show that, for the KV-algebra $(\mathrm{I\!R}^{2}, \mu)$,
$Ker\delta^{0}$ is the set of elements $\xi=(\xi_{1},\xi_{2}) \in
J(\mathrm{I\!R}^{2})$, satisfying the following equation
\begin{eqnarray*}
\begin{cases}
(\Gamma_{21}^{1}-\Gamma_{12}^{1})\xi_{2}=0\\
(\Gamma_{12}^{1}-\Gamma_{21}^{1})\xi_{1}=0\\
(\Gamma_{21}^{2}-\Gamma_{12}^{2})\xi_{2}=0\\
(\Gamma_{12}^{2}-\Gamma_{21}^{2})\xi_{1}=0\\
\end{cases}
\end{eqnarray*},  $Ker\delta^{1}$ is the set of linear applications $f_{1}\approx
\begin{pmatrix}
          \alpha&\beta\\\
          \gamma&\lambda
            \end{pmatrix}$ of $C^{1}(\mathrm{I\!R}^{2})$ satisfying the following equation
\begin{eqnarray*}
\begin{cases}
-\Gamma_{11}^{1}\alpha+\Gamma_{11}^{2}\beta-(\Gamma_{12}^{1}+\Gamma_{21}^{1})\gamma=0\\
(\Gamma_{12}^{2}-\Gamma_{11}^{1})\beta-\Gamma_{22}^{1}\gamma-\Gamma_{12}^{1}\lambda=0\\
(\Gamma_{11}^{2}-\Gamma_{11}^{1})\beta-\Gamma_{22}^{1}\gamma-\Gamma_{21}^{1}\lambda=0\\
\Gamma_{22}^{1}\alpha+(\Gamma_{22}^{2}-\Gamma_{12}^{1}-\Gamma_{21}^{1})\beta-2\Gamma_{22}^{1}\lambda=0\\
-2\Gamma_{11}^{2}\alpha+(\Gamma_{11}^{1}-\Gamma_{12}^{2}-\Gamma_{21}^{2})\gamma+\Gamma_{11}^{2}\lambda=0\\
-\Gamma_{12}^{2}\alpha-\Gamma_{11}^{2}\beta+(\Gamma_{12}^{1}-\Gamma_{22}^{2})\gamma=0\\
-\Gamma_{21}^{2}\alpha-\Gamma_{11}^{2}\beta+(\Gamma_{21}^{1}-\Gamma_{22}^{2})\gamma=0\\
-(\Gamma_{21}^{2}+\Gamma_{12}^{2})\beta+\Gamma_{22}^{1}\gamma-\Gamma_{22}^{2}\lambda=0\\
\end{cases}
\end{eqnarray*}, and $Ker\delta^{2}$ is the set of bilinear applications
$f_{2}\approx \biggl(\begin{pmatrix}
          e&f
          g&h
            \end{pmatrix},
        \begin{pmatrix}
         i&j
         k&l
            \end{pmatrix}\biggr)$ of $C^{2}(\mathrm{I\!R}^{2})$ verifying the following equation
\begin{eqnarray*}
\begin{cases}
(\Gamma_{12}^{1}-\Gamma_{21}^{1})e+\Gamma_{22}^{1}i+(\Gamma_{11}^{1}-\Gamma_{21}^{2})f+\Gamma_{21}^{1}j+(\Gamma_{12}^{2}-\Gamma_{11}^{1}
-\Gamma_{21}^{2})g-(\Gamma_{12}^{1}+\Gamma_{21}^{1})k+\Gamma_{11}^{2}h=0\\
\Gamma_{21}^{2}e+(\Gamma_{12}^{1}-2\Gamma_{21}^{1}+\Gamma_{22}^{2})i+\Gamma_{11}^{2}f-2\Gamma_{11}^{2}g+(\Gamma_{11}^{1}-2\Gamma_{21}^{2})k
+\Gamma_{11}^{2}l=0\\
-\Gamma_{22}^{1}e+(2\Gamma_{12}^{1}-\Gamma_{22}^{2})f+2\Gamma_{22}^{1}j-\Gamma_{22}^{1}k+(2\Gamma_{12}^{2}-\Gamma_{11}^{1}-\Gamma_{21}^{2})h
-\Gamma_{12}^{1}l=0\\
-\Gamma_{22}^{1}i+(\Gamma_{12}^{2}+\Gamma_{21}^{2})f+(\Gamma_{12}^{1}-\Gamma_{21}^{1}+\Gamma_{22}^{2})j-\Gamma_{12}^{1}g+(\Gamma_{12}^{1}-\Gamma_{22}^{2})k
-\Gamma_{11}^{2}h+(\Gamma_{12}^{2}-\Gamma_{21}^{2})l=0.
\end{cases}
\end{eqnarray*}
We show that, for the KV-algebra $(\mathrm{I\!R}^{2}, \mu)$,
$Im\delta^{-1}=\{0\}$; $Im\delta^{0}$ is the set of linear
applications $g\approx
\begin{pmatrix}
          u_{11}&u_{12}\\
          u_{21}&u_{22}
            \end{pmatrix}$ of $\mathrm{I\! R}^{2}$ satisfying the following equation
\begin{eqnarray*}
\begin{cases}
(\Gamma_{21}^{1}-\Gamma_{12}^{1})\xi_{2}=u_{11}\\
(\Gamma_{12}^{1}-\Gamma_{21}^{1})\xi_{1}=u_{12}\\
(\Gamma_{21}^{2}-\Gamma_{12}^{2})\xi_{2}=u_{21}\\
(\Gamma_{12}^{2}-\Gamma_{21}^{2})\xi_{1}=u_{22}.
\end{cases}
\end{eqnarray*}, and  $Im\delta^{1}$ is the set of bilinear applications $g\approx
\biggl(\begin{pmatrix}
          u_{11}&u_{12}\\
          u_{21}&u_{22}
            \end{pmatrix},
        \begin{pmatrix}
         v_{11}&v_{12}\\
         v_{221}&v_{22}
            \end{pmatrix}\biggr)$ of $\mathrm{I\!R}^{2}$ verifying the following  equation
\begin{eqnarray*}
\begin{cases}
-\Gamma_{11}^{1}\alpha+\Gamma_{11}^{2}\beta-(\Gamma_{12}^{1}+\Gamma_{21}^{1})\gamma=u_{11}\\
(\Gamma_{12}^{2}-\Gamma_{11}^{1})\beta-\Gamma_{22}^{1}\gamma-\Gamma_{12}^{1}\lambda=u_{12}\\
(\Gamma_{11}^{2}-\Gamma_{11}^{1})\beta-\Gamma_{22}^{1}\gamma-\Gamma_{21}^{1}\lambda=u_{21}\\
\Gamma_{22}^{1}\alpha+(\Gamma_{22}^{2}-\Gamma_{12}^{1}-\Gamma_{21}^{1})\beta-2\Gamma_{22}^{1}\lambda=u_{22}\\
-2\Gamma_{11}^{2}\alpha+(\Gamma_{11}^{1}-\Gamma_{12}^{2}-\Gamma_{21}^{2})\gamma+\Gamma_{11}^{2}\lambda=v_{11}\\
-\Gamma_{12}^{2}\alpha-\Gamma_{11}^{2}\beta+(\Gamma_{12}^{1}-\Gamma_{22}^{2})\gamma=v_{12}\\
-\Gamma_{21}^{2}\alpha-\Gamma_{11}^{2}\beta+(\Gamma_{21}^{1}-\Gamma_{22}^{2})\gamma=v_{21}\\
-(\Gamma_{21}^{2}+\Gamma_{12}^{2})\beta+\Gamma_{22}^{1}\gamma-\Gamma_{22}^{2}\lambda=v_{22}.
\end{cases}
\end{eqnarray*} In the same,  we show that, for $\mu \in Sol(\mathrm{I\!R}^{2}, KV)$,
be a KV-Hessian structure defined by
\begin{center}
  $\mu =\biggl(\begin{pmatrix}
          0&a\\
          a&0
            \end{pmatrix},
        \begin{pmatrix}
         b&0\\
         0&a
            \end{pmatrix}\biggr),   a\neq b \neq 0.$
\end{center}
We have, $ H_{KV}^{0}(\mu)\simeq \mathrm{I\!R}^{2}$, and
$H_{KV}^{1}(\mu)\simeq \{0\}$. We show that, for all  be a
KV-Hessian structure $\mu \in Sol(\mathrm{I\!R}^{2}, KV)$ defined by
\begin{center}
 $\mu =\biggl(\begin{pmatrix}
          0&a\\
          a&0
            \end{pmatrix},
        \begin{pmatrix}
         b&0\\
         0&a
            \end{pmatrix}\biggr),   a\neq b \neq 0$.
\end{center}
 we have $H_{KV}^{2}(\mu)\cong \biggl\{(E_{12},0),(E_{21},0),(0,E_{22})\biggr\}.$
For  all KV-algebra $(A,\mu)$, such  that $\mu_{t}$ be a deformation
of $\mu$; with $\mu_{t}=
\mu+t\nu_{1}+t^{2}\nu_{2}+t^{3}\nu_{3}+\dots$ we show that,
$\mu_{t}$ is a KV-structure $\Leftrightarrow$ $\nu_{i} \in
Ker\delta^{2}$, $i=1,2,3,\dots.$ We show that, for one  KV-Hessian
structure $\mu$ defined by
\begin{center}
  $\mu =\biggl(\begin{pmatrix}
          0&a\\
          a&0
            \end{pmatrix},
        \begin{pmatrix}
         b&0\\
         0&a
            \end{pmatrix}\biggr),   a\neq b \neq 0$
\end{center}
in a $(e_{1},e_{2})$ basis of $\mathrm{I\!R}^{2}$, the deformations
$\mu_{t}$ of the $\mu$ KV-structure are given by
\begin{eqnarray*}
        \biggl(\begin{pmatrix}
         a_{11}^{1}\sum_{i\geq 1}t^{i}  & a+ a_{12}^{1}\sum_{i\geq 1}t^{i}\\
           a+ a_{12}^{1}\sum_{i\geq 1}t^{i}&a_{22}^{1}\sum_{i\geq 1}t^{i}
            \end{pmatrix},
        \begin{pmatrix}
        b+ a_{11}^{2}\sum_{i\geq 1}t^{i}&(a_{21}^{2}+\frac{1}{2}a_{12}^{1})\sum_{i\geq 1}t^{i}\\
        (a_{21}^{2}+\frac{1}{2}a_{12}^{1})\sum_{i\geq 1}t^{i} &a+a_{12}^{1}\sum_{i\geq 1}t^{i}
            \end{pmatrix}\biggr)
\end{eqnarray*}
In section 2, We begin by recalling the basic concepts associated
with Koszul-Vinberg algebras. In section 3, we present the quantum
deformation. In section 4, we present the deformations of a Hessian
KV-structure on $\mathrm{I\!R}^{2}$. In section 5, the conclusion.

\section{Preliminaries}\label{sec2}
\subsection{Introduction to Koszul-Vinberg Algebras }
K is a commutative body of characteristic zero. Let A be a K-algebra; a, b and c elements of A.\\
- A multiplication in A is a K-bilinear application denoted by:
\begin{eqnarray*}
\mu :(a,b)\mapsto ab \in A
\end{eqnarray*}
- The associator of $\mu$ is the trilinear application defined by:
\begin{eqnarray*}
Ass_{\mu}: (a,b,c) \mapsto \mu(\mu(a,b),c)-\mu(a,\mu(b,c))\in A.
\end{eqnarray*}
-The KV-anomaly of $\mu$ is the trilinear application defined by:
\begin{eqnarray}
KV_{\mu}:(a,b,c) \mapsto Ass_{\mu}(a,b,c)-Ass_{\mu}(b,a,c)\in A.
\end{eqnarray}
\begin{definition}\cite{BOYOM1}
A K-algebra A is called a KV-algebra or Koszul-Vinberg algebra, if $ Ass_{\mu}(a,b,c)=Ass_{\mu}(b,a,c)$.\\
It means :  $KV_{\mu}=0$
\end{definition}
\subsubsection{KV-bimodule}
Let M be a vector space over K. For all a elements of A and x
elements of M, we define the following two K-bilinear applications:
\begin{center}
$\mu_{1}:(a,x)\mapsto ax \in M$ et $\mu_{2}:(x,a)\mapsto x a \in M$
\end{center}
For all b element for A, we note:
\begin{eqnarray*}
Ass(a,b,x)&=&(ab)x-a(bx)
\end{eqnarray*}
\begin{eqnarray*}
Ass(a,x,b)&=&(ax)b-a(xb)
\end{eqnarray*}
\begin{eqnarray*}
Ass(x,a,b)&=&(xa)b-x(ab)
\end{eqnarray*}
\begin{eqnarray*}
KV_{1}(a,b,x)&=&Ass(a,b,x)-Ass(b,a,x)
\end{eqnarray*}
\begin{eqnarray*}
KV_{2}(a,x,b)&=&Ass(a,x,b)-Ass(x,a,b)
\end{eqnarray*}
 \begin{definition}\cite{BOYOM3}
The K-vector space M equipped with the external operations $\mu_{1}$ and $\mu_{2}$ is called a KV-bimodule over A if the following operations are verified:\\
(i) $KV_{\mu_{1}}(a,b,x)= 0$\\
(ii) $KV_{\mu_{2}}(a,x,b)=0$\\
M is said to be a left-hand (resp. right-hand) KV-module if the
bilinear application $\mu_{2}$ (resp. $\mu_{1}$) is zero.
\end{definition}
\subsubsection{KV-cohomologie}
q is a positive integer. Let A be a KV-algebra and M a KV-bimodule
over A. Let $C^{q}(A,M)$ be the vector space of q-linear
applications over K from A to M. For all a elements of A and f
elements of $C^{q}(A,M)$, we define the following bilinear
applications
\begin{center}
$(a,f)\mapsto a.f \in C^{q}(A,M)$ et $(f,a)\mapsto f.a \in
C^{q}(A,M)$
\end{center}
$C^{q}(A,M)$ is a KV-bimodule on the KV-algebra A, for the following
actions of A on $C^{q}(A,M)$:
\begin{center}
$(af)(a_{1}, \dots,a_{q})=a(f(a_{1}, \dots,a_{q}))-\sum_{j=1}^{q}f(a_{1},\dots,aa_{j},\dots,a_{q})$\\
$(fa)(a_{1}, \dots,a_{q})=(f(a_{1}, \dots,a_{q}))a$,
\end{center}
where $(a_{1}, \dots,a_{q})$ is a element for $A^{q}$.\\
For each $\rho=1,\dots,q$, and f an element of $C^{q}(A,M)$, denote
$e_{\rho}(a)$, the following A-linear application:
\begin{center}
$e_{\rho}(a):f \mapsto e_{\rho}(a)f \in C^{q-1}(A,M)$
\end{center}
such that: $(e_{\rho}(a)f)(a_{1}, \dots.,a_{q-1})=f(a_{1},
\dots,a_{\rho-1},a,a_{\rho},\dots,a_{q-1})$,
for all $(a_{1}, \dots,a_{q-1})$ element of $A^{q-1}$.\\
We then define the cobord operator by: $\delta^{q}:f\mapsto
\delta^{q} f \in C^{q+1}(A,M)$
 such that:
\begin{eqnarray*}
\delta^{q} f(a_{1}\dots a_{q+1})=\Sigma_{1 \leq j\leq
q}(-1)^{j}\biggl\{(a_{j}f)(a_{1},\dots,\widehat{a}_{j}\dots,a_{q+1})
                                    +(e_{q}(a_{j})(fa_{q+1}))(a_{1}\dots\widehat{a}_{j}\dots\widehat{a}_{q+1})\biggr \},
\end{eqnarray*}
 where $(a_{1}, \dots,a_{q+1})$ is a element of $A^{q+1}$. We then check that: $\delta^{q+1} \circ \delta^{q} = 0$; and so $Im \delta^{q-1}\subset Ker \delta^{q}$.
Note $C_{KV}(A,M)=\oplus_{q\geq 0}C^{q}(A,M)$, we have a complex of
cochains:
\begin{eqnarray*}
 C^{0}(A,M) \stackrel{ \delta^{0} }{\longrightarrow} C^{1}(A,M)\stackrel{ \delta^{1} }{\longrightarrow}
 \dots C^{q}(A,M)\stackrel{ \delta^{q} }{\longrightarrow}C^{q+1}(A,M)\stackrel{ \delta^{q+1} }{\longrightarrow}C^{q+2}(A,M)\longrightarrow \dots
\end{eqnarray*}
\begin{definition}\cite{BOYOM1},\cite{BOYOM3}
 The cohomology of the KV complex $(C_{KV}(A,M),\delta)$
  is called the KV cohomology of the KV algebra A with values in the KV-bimodule M: $ H_{KV}(A,M)=\oplus_{q \geq 0}H_{KV}^{q}(A,M)$,
\end{definition}
where
\begin{equation}
 H_{KV}^{q}(A,M)=\frac{Ker \delta^{q}}{Im \delta^{q-1}}.
\end{equation}
 $\bullet$ The elements of $Ker \delta^{q}$ are called cocycles of degree q of the complex $(C_{KV}(A,M),\delta)$.\\
 $\bullet$ The elements of $Im \delta^{q-1}$ are called cobord of degree q of the complex $(C_{KV}(A,M),\delta)$.\\
 $\bullet$ the vector space $H_{KV}^{q}(A,M)$ is the $q^{th}$ Cohomology space of the complex $(C_{KV}(A,M),\delta)$.

\subsection{Description of a Hessian KV-structure}
A Hessian structure is a bilinear form that satisfies the symmetry
and non-degeneracy conditions. In geometry, Hessian structures are
often used to study the local properties of variants and the dynamic
behaviour of systems. In the context of Koszul-Vinberg algebra,
Hessian structures allow us to establish correspondences between
algebraic and geometric properties.
\begin{definition}\cite{Shima}
A $\mu$-KV-structure is Hessian if it is symmetrical and
non-degenerate.
\end{definition}

\section{Quantum deformation}\label{sec3}
Let $(A,\mu)$ be a KV-algebra, a deformation of $\mu$ is given by
$\mu_{t}= \mu+\sum_{s\geq1}t^{s}\nu_{s}$. For any algebra
$\mathrm{I\ ! R}^{2}$, $\mu)$, and for all $u=(x,y)$,
$v=(x^{'},y^{'})$, $w=(x^{''},y^{''})$ three vectors of $\mathrm{I\!
R}^{2}$. In a base $(e_{1},e_{2})$ de $\mathrm{I\! R}^{2}$, we have
$\mu(u,v)=\mu_{1}(u,v)e_{1}+\mu_{2}(u,v)e_{2}$. The matrix of the
bilinear forms $\mu_{1}$ and $\mu_{2}$, are given respectively by
\begin{center}
     $\Gamma_{1}= \begin{pmatrix}
          \Gamma_{11}^{1}&\Gamma_{12}^{1}\\
           \Gamma_{21}^{1}&\Gamma_{22}^{1}
            \end{pmatrix},
       \Gamma_{2}= \begin{pmatrix}
          \Gamma_{11}^{2}&\Gamma_{12}^{2}\\
           \Gamma_{21}^{2}&\Gamma_{22}^{2}
            \end{pmatrix}$
\end{center} We have the following proposition
\begin{proposition}
 $\mu$  is a KV-structure if and only if
\begin{eqnarray}
\begin{cases}\label{eq1}
 \Gamma_{21}^{1}(\Gamma_{12}^{2}-\Gamma_{11}^{1}-\Gamma_{21}^{2})+\Gamma_{12}^{1}(\Gamma_{11}^{1}-\Gamma_{21}^{2})
+\Gamma_{22}^{1}\Gamma_{11}^{2}=0 \\
\Gamma_{22}^{1}(2\Gamma_{12}^{2}-\Gamma_{11}^{1}-\Gamma_{21}^{2})+\Gamma_{12}^{1}(\Gamma_{12}^{1}-\Gamma_{22}^{2})=0\\
\Gamma_{11}^{2}(\Gamma_{12}^{1}-2\Gamma_{21}^{1}+\Gamma_{22}^{2})+\Gamma_{21}^{2}(\Gamma_{11}^{1}-\Gamma_{21}^{2})=0\\
\Gamma_{12}^{2}(\Gamma_{12}^{1}-\Gamma_{21}^{1}+\Gamma_{22}^{2})+\Gamma_{21}^{2}(\Gamma_{12}^{1}-\Gamma_{22}^{2})-\Gamma_{11}^{2}\Gamma_{22}^{1}=0\\
\end{cases}
\end{eqnarray}
\end{proposition}
\begin{proof}
See in \cite{Mopeng}.
\end{proof}

Let $Sol(\mathrm{I\! R}^{2}, KV)$ be the set of solutions of the
equation Eq.~(\ref{eq1}). Let $\mu$ be an element of
$Sol(\mathrm{I\! R}^{2}, KV)$. We have the following classification
of $\mu$:
\subsubsection*{Classe 1: $\mu$ is degenerate}
 \begin{eqnarray*}
\mu= \biggl(\begin{pmatrix}
         0 &0\\
          \Gamma_{21}^{1} &\Gamma_{22}^{1}
            \end{pmatrix},
        \begin{pmatrix}
          0&0\\
           0&\Gamma_{22}^{2}
            \end{pmatrix}\biggr);
\mu=\biggl(\begin{pmatrix}
         0&\Gamma_{22}^{2}\\
          0 &\Gamma_{22}^{2}
            \end{pmatrix},
        \begin{pmatrix}
         0&0\\
           0&\Gamma_{22}^{2}
            \end{pmatrix}\biggr);
\mu=\biggl(\begin{pmatrix}
         0 &0\\
          \Gamma_{21}^{1}&\Gamma_{22}^{1}
            \end{pmatrix},
        \begin{pmatrix}
         0&0 \\
         0&0
            \end{pmatrix}\biggr);
\end{eqnarray*}
\begin{eqnarray*}
\mu=\biggl(\begin{pmatrix}
         \Gamma_{11}^{1} &0\\
          \Gamma_{22}^{2}&0
            \end{pmatrix},
        \begin{pmatrix}
         0 & \Gamma_{11}^{1}\\
           0&\Gamma_{22}^{2}
        \end{pmatrix}\biggr);
\mu=\biggl(\begin{pmatrix}
         \Gamma_{11}^{1}&\Gamma_{22}^{2}\\
          0&0
            \end{pmatrix},
        \begin{pmatrix}
         0&0 \\
           \Gamma_{11}^{1}&\Gamma_{22}^{2}
            \end{pmatrix}\biggr);
\mu=\biggl(\begin{pmatrix}
          \Gamma_{11}^{1}&0\\
          0 &0
            \end{pmatrix},
        \begin{pmatrix}
         0 &0 \\
           \Gamma_{11}^{1}&0
            \end{pmatrix}\biggr);
\end{eqnarray*}
\begin{eqnarray*}
\mu=\biggl(\begin{pmatrix}
         0 &0\\
          0 &0
            \end{pmatrix},
        \begin{pmatrix}
         0 & \Gamma_{12}^{2}\\
           0&0
        \end{pmatrix}\biggr);
\mu=\biggl(\begin{pmatrix}
         0 &0\\
          \Gamma_{21}^{1}&0
            \end{pmatrix},
        \begin{pmatrix}
         0  & 0\\
           0&0
            \end{pmatrix}\biggr);
\mu=\biggl(\begin{pmatrix}
         \Gamma_{11}^{1} &0\\
          0 &0
            \end{pmatrix},
        \begin{pmatrix}
        \Gamma_{11}^{2} &\Gamma_{12}^{2} \\
           0&0
            \end{pmatrix}\biggr);
\end{eqnarray*}
\begin{eqnarray*}
\mu=\biggl(\begin{pmatrix}
         \Gamma_{11}^{1} &0\\
          0 &0
            \end{pmatrix},
        \begin{pmatrix}
        0 & \Gamma_{12}^{2}\\
           0&0
        \end{pmatrix}\biggr);
\mu=\biggl(\begin{pmatrix}
         0 &\Gamma_{22}^{2}\\
          0 &0
            \end{pmatrix},
        \begin{pmatrix}
         0  & 0\\
           0&\Gamma_{22}^{2}
            \end{pmatrix}\biggr);
\mu=\biggl(\begin{pmatrix}
         0 &0\\
          \Gamma_{21}^{1} &0
            \end{pmatrix},
        \begin{pmatrix}
          0 & 0\\
           0&\Gamma_{22}^{2}
            \end{pmatrix}\biggr);
\end{eqnarray*}
\begin{eqnarray*}
\mu=\biggl(\begin{pmatrix}
         0 &0\\
          \Gamma_{21}^{1} &0
            \end{pmatrix},
        \begin{pmatrix}
        \Gamma_{11}^{2} & \Gamma_{12}^{2}\\
           0&0
        \end{pmatrix}\biggr);
\mu=\biggl(\begin{pmatrix}
         0&0\\
          0&0
            \end{pmatrix},
        \begin{pmatrix}
        \Gamma_{11}^{2}&\Gamma_{12}^{2}  \\
         0&0
            \end{pmatrix}\biggr).
\end{eqnarray*}
\subsubsection*{Classe 2: $\mu$  is degenerate and symmetrical}
 \begin{eqnarray*}
\mu= \biggl(\begin{pmatrix}
         0 &0\\
          0 &\Gamma_{22}^{1}
            \end{pmatrix},
        \begin{pmatrix}
          0&0\\
           0&\Gamma_{22}^{2}
            \end{pmatrix}\biggr);
\mu=\biggl(\begin{pmatrix}
         0&0\\
          0 &\Gamma_{22}^{1}
            \end{pmatrix},
        \begin{pmatrix}
         0&0\\
           0&0
            \end{pmatrix}\biggr);
\mu=\biggl(\begin{pmatrix}
         0 &0\\
          0&0
            \end{pmatrix},
        \begin{pmatrix}
         0&0 \\
         0&\Gamma_{22}^{2}
            \end{pmatrix}\biggr);
\end{eqnarray*}
\begin{eqnarray*}
\mu=\biggl(\begin{pmatrix}
         \Gamma_{11}^{1} &0\\
          0&0
            \end{pmatrix},
        \begin{pmatrix}
         0&0 \\
         0&\Gamma_{22}^{2}
            \end{pmatrix}\biggr);
\mu=\biggl(\begin{pmatrix}
         0&0\\
          0&0
            \end{pmatrix},
        \begin{pmatrix}
         \Gamma_{11}^{2} &0 \\
         0&0
            \end{pmatrix}\biggr).
\end{eqnarray*}
\subsubsection*{Classe 3: $\mu$  is non-degenerate }
 \begin{eqnarray*}
\mu=\biggl(\begin{pmatrix}
         \Gamma_{11}^{1} &\Gamma_{22}^{2}\\
          2\Gamma_{22}^{2} &0
            \end{pmatrix},
        \begin{pmatrix}
         0  &2\Gamma_{11}^{1} \\
           \Gamma_{11}^{1}&\Gamma_{22}^{2}
            \end{pmatrix}\biggr);
\mu=\biggl(\begin{pmatrix}
         \Gamma_{11}^{1} &\Gamma_{12}^{1}\\
         \Gamma_{21}^{1} &\Gamma_{22}^{1}
            \end{pmatrix},
        \begin{pmatrix}
         \Gamma_{11}^{2}  & \Gamma_{12}^{2}\\
           \Gamma_{21}^{2}&\Gamma_{22}^{2}
            \end{pmatrix}\biggr),
\end{eqnarray*}
then $\Gamma_{12}^{1}\neq \Gamma_{21}^{1}, \Gamma_{12}^{2}\neq
\Gamma_{21}^{2}$ and the others fixed.
\subsubsection*{Classe 4: $\mu$ is non-degenerate and symmetrical}
 \begin{eqnarray*}
\mu=\biggl(\begin{pmatrix}
         0&\Gamma_{22}^{2}\\
          \Gamma_{22}^{2} &0
            \end{pmatrix},
        \begin{pmatrix}
         \Gamma_{11}^{2}  &0 \\
           0&\Gamma_{22}^{2}
            \end{pmatrix}\biggr);
\mu=\biggl(\begin{pmatrix}
         \Gamma_{11}^{1} &0\\
          0 &\Gamma_{22}^{1}
            \end{pmatrix},
        \begin{pmatrix}
         0  &\Gamma_{11}^{1} \\
           \Gamma_{11}^{1}&\Gamma_{22}^{2}
            \end{pmatrix}\biggr);
\mu=\biggl(\begin{pmatrix}
         \Gamma_{11}^{1} &0\\
          0 &\Gamma_{22}^{1}
            \end{pmatrix},
        \begin{pmatrix}
         \Gamma_{11}^{2}  &0 \\
           0&\Gamma_{22}^{2}
            \end{pmatrix}\biggr);
\end{eqnarray*}
\begin{eqnarray*}
\mu=\biggl(\begin{pmatrix}
         \Gamma_{11}^{1} &\Gamma_{22}^{2}\\
          \Gamma_{22}^{2} &0
            \end{pmatrix},
        \begin{pmatrix}
         \Gamma_{11}^{2}  &0 \\
           0&\Gamma_{22}^{2}
            \end{pmatrix}\biggr);
\mu=\biggl(\begin{pmatrix}
         \Gamma_{11}^{1} &0\\
          0&\Gamma_{22}^{1}
            \end{pmatrix},
        \begin{pmatrix}
         0&\Gamma_{11}^{1} \\
           \Gamma_{11}^{1}&0
            \end{pmatrix}\biggr).
\end{eqnarray*}
\subsubsection*{Classe 5: $\mu$ is mixed}
 \begin{eqnarray*}
\mu= \biggl(\begin{pmatrix}
         0 &\Gamma_{22}^{2}\\
          \Gamma_{21}^{1} &0
            \end{pmatrix},
        \begin{pmatrix}
          0&0\\
           0&\Gamma_{22}^{2}
            \end{pmatrix}\biggr);
\mu=\biggl(\begin{pmatrix}
         0 &\Gamma_{22}^{2}\\
          \Gamma_{21}^{1} &\Gamma_{22}^{1}
            \end{pmatrix},
        \begin{pmatrix}
         0 & 0\\
          0 &\Gamma_{22}^{2}
            \end{pmatrix}\biggr);
\mu=\biggl(\begin{pmatrix}
         2\Gamma_{12}^{2} &\Gamma_{22}^{2}\\
          2\Gamma_{22}^{2} &0
            \end{pmatrix},
        \begin{pmatrix}
         0 & \Gamma_{12}^{2}\\
         0&\Gamma_{22}^{2}
            \end{pmatrix}\biggr).
\end{eqnarray*}
\begin{eqnarray*}
\mu= \biggl(\begin{pmatrix}
         \Gamma_{11}^{1} &0\\
          0 &0
            \end{pmatrix},
        \begin{pmatrix}
          \Gamma_{11}^{2}&\Gamma_{12}^{2}\\
           \Gamma_{11}^{1}&0
            \end{pmatrix}\biggr);
\mu=\biggl(\begin{pmatrix}
         \Gamma_{11}^{1} &0\\
          0 &0
            \end{pmatrix},
        \begin{pmatrix}
         0 & \Gamma_{12}^{2}\\
          \Gamma_{11}^{1} &0
            \end{pmatrix}\biggr);
\mu=\biggl(\begin{pmatrix}
         \Gamma_{11}^{1} &0\\
          \Gamma_{21}^{1} &0
            \end{pmatrix},
        \begin{pmatrix}
         0 & 2\Gamma_{11}^{1}\\
          \Gamma_{11}^{1} &2\Gamma_{21}^{1}
            \end{pmatrix}\biggr).
\end{eqnarray*}
\subsubsection*{Classe 6: $\mu$ is mixed and symmetrical}
 \begin{eqnarray*}
\mu= \biggl(\begin{pmatrix}
         \Gamma_{11}^{1} &0\\
          0 &0
            \end{pmatrix},
        \begin{pmatrix}
          0&\Gamma_{11}^{1}\\
           \Gamma_{11}^{1}&\Gamma_{22}^{2}
            \end{pmatrix}\biggr);
\mu=\biggl(\begin{pmatrix}
         \Gamma_{11}^{1} &\Gamma_{22}^{2}\\
          \Gamma_{22}^{2} &0
            \end{pmatrix},
        \begin{pmatrix}
         0 & 0\\
          0 &\Gamma_{22}^{2}
            \end{pmatrix}\biggr);
\mu= \biggl(\begin{pmatrix}
         \Gamma_{11}^{1} &0\\
          0 &\Gamma_{22}^{1}
            \end{pmatrix},
        \begin{pmatrix}
          \Gamma_{11}^{2}&0\\
           0&0
            \end{pmatrix}\biggr);
\end{eqnarray*}
\begin{eqnarray*}
\mu=\biggl(\begin{pmatrix}
         0 &0\\
          0 &\Gamma_{22}^{1}
            \end{pmatrix},
        \begin{pmatrix}
         \Gamma_{11}^{2} & 0\\
          0 &\Gamma_{22}^{2}
            \end{pmatrix}\biggr)
\end{eqnarray*}
The following proposition characterizes Hessian KV-structures on
$\mathrm{I\!R}^{2}$.
\begin{proposition}
On $\mathrm{I\!R}^{2}$, the following KV-structures are Hessian:
 \begin{eqnarray*}
\mu=\biggl(\begin{pmatrix}
         0&\Gamma_{22}^{2}\\
          \Gamma_{22}^{2} &0
            \end{pmatrix},
        \begin{pmatrix}
         \Gamma_{11}^{2}  &0 \\
           0&\Gamma_{22}^{2}
            \end{pmatrix}\biggr);
\mu=\biggl(\begin{pmatrix}
         \Gamma_{11}^{1} &0\\
          0 &\Gamma_{22}^{1}
            \end{pmatrix},
        \begin{pmatrix}
         0  &\Gamma_{11}^{1} \\
           \Gamma_{11}^{1}&\Gamma_{22}^{2}
            \end{pmatrix}\biggr);
\mu=\biggl(\begin{pmatrix}
         \Gamma_{11}^{1} &0\\
          0 &\Gamma_{22}^{1}
            \end{pmatrix},
        \begin{pmatrix}
         \Gamma_{11}^{2}  &0 \\
           0&\Gamma_{22}^{2}
            \end{pmatrix}\biggr);
\end{eqnarray*}
\begin{eqnarray*}
\mu=\biggl(\begin{pmatrix}
         \Gamma_{11}^{1} &\Gamma_{22}^{2}\\
          \Gamma_{22}^{2} &0
            \end{pmatrix},
        \begin{pmatrix}
         \Gamma_{11}^{2}  &0 \\
           0&\Gamma_{22}^{2}
            \end{pmatrix}\biggr);
\mu=\biggl(\begin{pmatrix}
         \Gamma_{11}^{1} &0\\
          0&\Gamma_{22}^{1}
            \end{pmatrix},
        \begin{pmatrix}
         0&\Gamma_{11}^{1} \\
           \Gamma_{11}^{1}&0
            \end{pmatrix}\biggr).
\end{eqnarray*}
\end{proposition}
\begin{proof}
$\bullet$ If $\mu(u,v)=\mu(v,u)$, then we have:
$(u^{t}\Gamma_{1}v,u^{t}\Gamma_{2}v)=(v^{t}\Gamma_{1}u,v^{t}\Gamma_{2}u).$\\
This implies that: $ \begin{cases}
u^{t}\Gamma_{1}v=v^{t}\Gamma_{1}u\\
u^{t}\Gamma_{2}v=v^{t}\Gamma_{2}u\\
     \end{cases}$.
So $\Gamma_{1}$ et $\Gamma_{2}$ are symmetrical.\\
$\bullet$ If $\mu(u,v)=0$, then we have:
 $\begin{cases}
u^{t}\Gamma_{1}v=0\\
u^{t}\Gamma_{2}v=0\\
     \end{cases}$.
This implies that: $ \begin{cases}
u=0\\
v=0\\
 \end{cases}$.
So $\Gamma_{1}$ et $\Gamma_{2}$ are degenerated.\\
Since, if $\mu$ is non-zero, then $\Gamma_{1}$ and $\Gamma_{2}$ are
non-degenerate.
\end{proof}
\subsection{KV-Cohomologie on  $\mathbb{R}^{2}$}
In this section, we recall the KV-cohomology on a 2-dimensional vector space. This will enable us to define the Hessian case.\\
Let the KV-algebra $(\mathrm{I\!R}^{2}, \mu)$; the $q-th$ cohomology
space of the complex $(C_{KV}(\mu),\delta)$ is defined by:
\begin{equation}
 H_{KV}^{q}(\mu)\cong\frac{Ker \delta^{q}}{Im \delta^{q-1}}.
\end{equation}
The following lemma characterizes $Ker\delta^{q}, q= 0, 1,2$,  on
$\mathrm{I\! R}^{2}$, to define the first KV-cohomology spaces
$H_{KV}^{0}(\mu), H_{KV}^{1}(\mu), H_{KV}^{2}(\mu)$.
\begin{proposition}
For the KV-algebra $(\mathrm{I\! R}^{2}, \mu)$, we have:\\
(i)$Ker\delta^{0}$ is the set of elements $\xi=(\xi_{1},\xi_{2}) \in
J(\mathrm{I\!R}^{2})$, satisfying the following equation
Eq.~(\ref{eq2})
\begin{eqnarray}
\begin{cases}\label{eq2}
(\Gamma_{21}^{1}-\Gamma_{12}^{1})\xi_{2}=0\\
(\Gamma_{12}^{1}-\Gamma_{21}^{1})\xi_{1}=0\\
(\Gamma_{21}^{2}-\Gamma_{12}^{2})\xi_{2}=0\\
(\Gamma_{12}^{2}-\Gamma_{21}^{2})\xi_{1}=0\\
\end{cases}
\end{eqnarray}
(ii)$Ker\delta^{1}$ is the set of linear applications $f_{1}\approx
\begin{pmatrix}
          \alpha&\beta\\\
          \gamma&\lambda
            \end{pmatrix}$ of $C^{1}(\mathrm{I\!R}^{2})$ satisfying the following equation Eq.~(\ref{eq3})
\begin{eqnarray}
\begin{cases}\label{eq3}
-\Gamma_{11}^{1}\alpha+\Gamma_{11}^{2}\beta-(\Gamma_{12}^{1}+\Gamma_{21}^{1})\gamma=0\\
(\Gamma_{12}^{2}-\Gamma_{11}^{1})\beta-\Gamma_{22}^{1}\gamma-\Gamma_{12}^{1}\lambda=0\\
(\Gamma_{11}^{2}-\Gamma_{11}^{1})\beta-\Gamma_{22}^{1}\gamma-\Gamma_{21}^{1}\lambda=0\\
\Gamma_{22}^{1}\alpha+(\Gamma_{22}^{2}-\Gamma_{12}^{1}-\Gamma_{21}^{1})\beta-2\Gamma_{22}^{1}\lambda=0\\
-2\Gamma_{11}^{2}\alpha+(\Gamma_{11}^{1}-\Gamma_{12}^{2}-\Gamma_{21}^{2})\gamma+\Gamma_{11}^{2}\lambda=0\\
-\Gamma_{12}^{2}\alpha-\Gamma_{11}^{2}\beta+(\Gamma_{12}^{1}-\Gamma_{22}^{2})\gamma=0\\
-\Gamma_{21}^{2}\alpha-\Gamma_{11}^{2}\beta+(\Gamma_{21}^{1}-\Gamma_{22}^{2})\gamma=0\\
-(\Gamma_{21}^{2}+\Gamma_{12}^{2})\beta+\Gamma_{22}^{1}\gamma-\Gamma_{22}^{2}\lambda=0\\
\end{cases}
\end{eqnarray}
(iii) $Ker\delta^{2}$ is the set of bilinear applications
$f_{2}\approx \biggl(\begin{pmatrix}
          e&f
          g&h
            \end{pmatrix},
        \begin{pmatrix}
         i&j
         k&l
            \end{pmatrix}\biggr)$ of $C^{2}(\mathrm{I\!R}^{2})$ verifying the following equation Eq.~(\ref{eq4})
\begin{eqnarray}
\begin{cases} \label{eq4}
(\Gamma_{12}^{1}-\Gamma_{21}^{1})e+\Gamma_{22}^{1}i+(\Gamma_{11}^{1}-\Gamma_{21}^{2})f+\Gamma_{21}^{1}j+(\Gamma_{12}^{2}-\Gamma_{11}^{1}
-\Gamma_{21}^{2})g-(\Gamma_{12}^{1}+\Gamma_{21}^{1})k+\Gamma_{11}^{2}h=0\\
\Gamma_{21}^{2}e+(\Gamma_{12}^{1}-2\Gamma_{21}^{1}+\Gamma_{22}^{2})i+\Gamma_{11}^{2}f-2\Gamma_{11}^{2}g+(\Gamma_{11}^{1}-2\Gamma_{21}^{2})k
+\Gamma_{11}^{2}l=0\\
-\Gamma_{22}^{1}e+(2\Gamma_{12}^{1}-\Gamma_{22}^{2})f+2\Gamma_{22}^{1}j-\Gamma_{22}^{1}k+(2\Gamma_{12}^{2}-\Gamma_{11}^{1}-\Gamma_{21}^{2})h
-\Gamma_{12}^{1}l=0\\
-\Gamma_{22}^{1}i+(\Gamma_{12}^{2}+\Gamma_{21}^{2})f+(\Gamma_{12}^{1}-\Gamma_{21}^{1}+\Gamma_{22}^{2})j-\Gamma_{12}^{1}g+(\Gamma_{12}^{1}-\Gamma_{22}^{2})k
-\Gamma_{11}^{2}h+(\Gamma_{12}^{2}-\Gamma_{21}^{2})l=0\\
\end{cases}
\end{eqnarray}
\end{proposition}
\begin{proof} Consider the following KV-complex
\begin{center}
$0 \stackrel{ \delta^{-1} }{\longrightarrow} C^{0}(\mathbb{R}^{2})
\stackrel{ \delta^{0} }{\longrightarrow} C^{1}(\mathbb{R}^{2})
\stackrel{ \delta^{1} }{\longrightarrow}
C^{2}(\mathbb{R}^{2})\stackrel{ \delta^{2} }{\longrightarrow}
\ldots,$
\end{center}
where $C^{0}(\mathrm{I\! R}^{2})=J(\mathrm{I\! R}^{2})=\{\xi \in
\mathrm{I\! R}^{2}, Ass_{\mu}(u,v,\xi)=0, u, v \in \mathrm{I\!
R}^{2}\}$. We have\\
(i) $\delta^{0}\xi \mapsto \delta^{0}(\xi)\in
C^{1}(\mathrm{I\!R}^{2}),$ such that, for all element $u=(x,y)$ of
$\mathrm{I\!R}^{2}$
 \begin{eqnarray*}
\delta^{0}(\xi)(u)&=& -u\xi + \xi u ,\forall \xi=(\xi_{1},\xi_{2}) \in J(\mathbb{R}^{2})\\
                  &=& -\mu(u,\xi)+\mu(\xi,\mu)\\
                  &=& (-u^{t}\Gamma_{1}\xi+\xi^{t}\Gamma_{1}u,-u^{t}\Gamma_{2}\xi+\xi^{t}\Gamma_{2}u)\\
                  &=& \biggl((\Gamma_{21}^{1}-\Gamma_{12}^{1})x\xi_{2}+(\Gamma_{12}^{1}-\Gamma_{21}^{1})y\xi_{1},
                   (\Gamma_{21}^{2}-\Gamma_{12}^{2})x\xi_{2}+(\Gamma_{12}^{2}-\Gamma_{21}^{2})y\xi_{1}\biggr)
\end{eqnarray*}
So, $\delta^{0}(\xi)(u)= 0$ if only if $\begin{cases}
(\Gamma_{21}^{1}-\Gamma_{12}^{1})\xi_{2}=0\\
(\Gamma_{12}^{1}-\Gamma_{21}^{1})\xi_{1}=0\\
(\Gamma_{21}^{2}-\Gamma_{12}^{2})\xi_{2}=0\\
(\Gamma_{12}^{2}-\Gamma_{21}^{2})\xi_{1}=0\\
\end{cases}$.\\
    (ii)] $\delta^{1}:f_{1} \mapsto \delta^{1}f_{1}\in
C^{2}(\mathrm{I\!R}^{2}),$ such that, for any $ u=(x,y),
v=(x',y')\in \mathrm{I\!R}^{2}$, we have
\begin{eqnarray*}
\delta^{1}f_{1}(u,v)&=& -uf_{1}(v)+f_{1}(uv)-f_{1}(u)v\\
                      &=& -\mu(u,f_{1}(v))+f_{1}(\mu(u,v))-\mu(f_{1}(u),v);
\end{eqnarray*}
where $f_{1}$ linear application of $\mathrm{I\!R}^{2}$ in
$\mathrm{I\!R}^{2}$ associated at one matrix $A=\begin{pmatrix}
    \alpha& \beta\\
    \gamma&\lambda
     \end{pmatrix}$,
    in the basis $(e_{1}, e_{2})$ de $\mathrm{I\!R}^{2}$.
So, $f_{1}(u)=(\alpha x+\beta y,\gamma x+\lambda y)$,  we  have\\
$\bullet$
$uf_{1}(v)=(u^{t}\Gamma_{1}f_{1}(v),u^{t}\Gamma_{2}f_{1}(v))$, with
\begin{eqnarray*}
u^{t}\Gamma_{1}f_{1}(v)&=&(\alpha\Gamma_{11}^{1}+\gamma\Gamma_{12}^{1})xx'+(\beta\Gamma_{11}^{1}+\lambda\Gamma_{12}^{1})xy'\\
                        &+&(\alpha\Gamma_{21}^{1}+\gamma\Gamma_{22}^{1})x'y+(\beta\Gamma_{21}^{1}+\lambda\Gamma_{22}^{1})yy',
\end{eqnarray*}
\begin{eqnarray*}
u^{t}\Gamma_{2}f_{1}(v)&=&(\alpha\Gamma_{11}^{2}+\gamma\Gamma_{12}^{2})xx'+ (\beta\Gamma_{11}^{2}+\lambda\Gamma_{12}^{2})xy'\\
                       &+&(\alpha\Gamma_{21}^{2}+\gamma\Gamma_{22}^{2})x'y+(\beta\Gamma_{21}^{2}+\lambda\Gamma_{22}^{2}) yy'
\end{eqnarray*}
then  $ uf_{1}(v)=
\begin{cases}\biggl((\alpha\Gamma_{11}^{1}+\gamma\Gamma_{12}^{1})xx'+(\beta\Gamma_{11}^{1}+\lambda\Gamma_{12}^{1})xy'\\
+(\alpha\Gamma_{21}^{1}+\gamma\Gamma_{22}^{1})x'y+(\beta\Gamma_{21}^{1}+\lambda\Gamma_{22}^{1})yy',\\
(\alpha\Gamma_{11}^{2}+\gamma\Gamma_{12}^{2})xx'+ (\beta\Gamma_{11}^{2}+\lambda\Gamma_{12}^{2})xy'\\
+(\alpha\Gamma_{21}^{2}+\gamma\Gamma_{22}^{2})x'y+(\beta\Gamma_{21}^{2}+\lambda\Gamma_{22}^{2}) yy' \biggl)\\
\end{cases}$\\
$\bullet$\begin{eqnarray*}
f_{1}(uv)&=&f_{1}(\Gamma_{11}^{1}xx^{'}+\Gamma_{12}^{1}xy^{'}+\Gamma_{21}^{1}x^{'}y+\Gamma_{22}^{1}yy^{'},
        \Gamma_{11}^{2}xx^{'}+\Gamma_{12}^{2}xy^{'}+\Gamma_{21}^{2}x^{'}y+\Gamma_{22}^{2}yy^{'})\\
        &=& A .(\Gamma_{11}^{1}xx^{'}+\Gamma_{12}^{1}xy^{'}+\Gamma_{21}^{1}x^{'}y+\Gamma_{22}^{1}yy^{'},
        \Gamma_{11}^{2}xx^{'}+\Gamma_{12}^{2}xy^{'}+\Gamma_{21}^{2}x^{'}y+\Gamma_{22}^{2}yy^{'})
\end{eqnarray*}
Then $f_{1}(uv)=
\begin{cases}\biggl((\alpha\Gamma_{11}^{1}+\beta\Gamma_{11}^{2})xx'+(\alpha\Gamma_{12}^{1}+\beta\Gamma_{12}^{2})xy'\\
+(\alpha\Gamma_{21}^{1}+\beta\Gamma_{21}^{2})x'y+(\alpha\Gamma_{22}^{1}+\beta\Gamma_{22}^{2})yy',\\
(\gamma\Gamma_{11}^{1}+\lambda\Gamma_{11}^{2})xx'+(\gamma\Gamma_{12}^{1}+\lambda\Gamma_{12}^{2})xy'\\
+(\gamma\Gamma_{21}^{1}+\lambda\Gamma_{21}^{2})x'y+(\gamma\Gamma_{22}^{1}+\lambda\Gamma_{22}^{2})yy' \biggr)\\
\end{cases}$;\\
$\bullet$
$f_{1}(u)v=(f_{1}(u)^{t}\Gamma_{1}v,f_{1}(u)^{t}\Gamma_{2}v)$, with
\begin{eqnarray*}
f_{1}(u)^{t}\Gamma_{1}v &=&(\alpha\Gamma_{11}^{1}+\gamma\Gamma_{21}^{1})xx'+(\alpha\Gamma_{12}^{1}+\gamma\Gamma_{22}^{1})xy'\\
                        &+&(\beta\Gamma_{11}^{1}+\lambda\Gamma_{21}^{1})x'y+(\beta\Gamma_{12}^{1}+\lambda\Gamma_{22}^{1})yy',
\end{eqnarray*}
\begin{eqnarray*}
f_{1}(u)^{t}\Gamma_{2}v &=&(\alpha\Gamma_{11}^{2}+\gamma\Gamma_{21}^{2})xx'+ (\alpha\Gamma_{12}^{2}+\gamma\Gamma_{22}^{2})xy'\\
                        &+&(\beta\Gamma_{11}^{2}+\lambda\Gamma_{21}^{2})x'y+(\beta\Gamma_{12}^{2}+\lambda\Gamma_{22}^{2}) yy'
\end{eqnarray*}
Then $ f_{1}(u)v=
\begin{cases}\biggl((\alpha\Gamma_{11}^{1}+\gamma\Gamma_{21}^{1})xx'+(\alpha\Gamma_{12}^{1}+\gamma\Gamma_{22}^{1})xy'\\
+(\beta\Gamma_{11}^{1}+\lambda\Gamma_{21}^{1})x'y+(\beta\Gamma_{12}^{1}+\lambda\Gamma_{22}^{1})yy',\\
(\alpha\Gamma_{11}^{2}+\gamma\Gamma_{21}^{2})xx'+ (\alpha\Gamma_{12}^{2}+\gamma\Gamma_{22}^{2})xy'\\
+(\beta\Gamma_{11}^{2}+\lambda\Gamma_{21}^{2})x'y+(\beta\Gamma_{12}^{2}+\lambda\Gamma_{22}^{2}) yy'\biggr )\\
\end{cases}$.\\
We obtain the expression of $\delta^{1}f_{1}(u,v)$:\\
\begin{eqnarray}\delta^{1}f_{1}(u,v)=
\begin{cases}\biggl((-\Gamma_{11}^{1}\alpha+\Gamma_{11}^{2}\beta-(\Gamma_{12}^{1}+\Gamma_{21}^{1})\gamma)xx'+((\Gamma_{12}^{2}-\Gamma_{11}^{1})\beta\\
-\Gamma_{22}^{1}\gamma-\Gamma_{12}^{1}\lambda)xy'+((\Gamma_{11}^{2}-\Gamma_{11}^{1})\beta-\Gamma_{22}^{1}\gamma-\Gamma_{21}^{1}\lambda)x'y\\
+(\Gamma_{22}^{1}\alpha+(\Gamma_{22}^{2}-\Gamma_{12}^{1}-\Gamma_{21}^{1})\beta-2\Gamma_{22}^{1}\lambda)yy',\\
(-2\Gamma_{11}^{2}\alpha+(\Gamma_{11}^{1}-\Gamma_{12}^{2}-\Gamma_{21}^{2})\gamma+\Gamma_{11}^{2}\lambda)xx'+(-\Gamma_{12}^{2}\alpha-\Gamma_{11}^{2}\beta\\
+(\Gamma_{12}^{1}-\Gamma_{22}^{2})\gamma)xy'+(-\Gamma_{21}^{2}\alpha-\Gamma_{11}^{2}\beta+(\Gamma_{21}^{1}-\Gamma_{22}^{2})\gamma)x'y\\
+(-(\Gamma_{21}^{2}+\Gamma_{12}^{2})\beta+\Gamma_{22}^{1}\gamma-\Gamma_{22}^{2}\lambda)yy' \biggr)\\
\end{cases}.
\end{eqnarray}
Then
\begin{eqnarray*}
f_{1}\in Ker\delta^{1} &\Longleftrightarrow & \delta^{1}f_{1}(u,v)=0\\
                       &\Longleftrightarrow & \delta^{1}f_{1}(e_{1},e_{j})=0, 1\leq i, j \leq 2 \\
                       &\Longleftrightarrow & \begin{cases}
-\Gamma_{11}^{1}\alpha+\Gamma_{11}^{2}\beta-(\Gamma_{12}^{1}+\Gamma_{21}^{1})\gamma=0\\
(\Gamma_{12}^{2}-\Gamma_{11}^{1})\beta-\Gamma_{22}^{1}\gamma-\Gamma_{12}^{1}\lambda=0\\
(\Gamma_{11}^{2}-\Gamma_{11}^{1})\beta-\Gamma_{22}^{1}\gamma-\Gamma_{21}^{1}\lambda=0\\
\Gamma_{22}^{1}\alpha+(\Gamma_{22}^{2}-\Gamma_{12}^{1}-\Gamma_{21}^{1})\beta-2\Gamma_{22}^{1}\lambda=0\\
-2\Gamma_{11}^{2}\alpha+(\Gamma_{11}^{1}-\Gamma_{12}^{2}-\Gamma_{21}^{2})\gamma+\Gamma_{11}^{2}\lambda=0\\
-\Gamma_{12}^{2}\alpha-\Gamma_{11}^{2}\beta+(\Gamma_{12}^{1}-\Gamma_{22}^{2})\gamma=0\\
-\Gamma_{21}^{2}\alpha-\Gamma_{11}^{2}\beta+(\Gamma_{21}^{1}-\Gamma_{22}^{2})\gamma=0\\
-(\Gamma_{21}^{2}+\Gamma_{12}^{2})\beta+\Gamma_{22}^{1}\gamma-\Gamma_{22}^{2}\lambda=0\\
\end{cases}
\end{eqnarray*}\\
    (iii) $\delta^{2}:f_{2} \mapsto \delta^{2}f_{2}\in
C^{3}(\mathrm{I\!R}^{2}),$ such that, for all $u=(x,y), v=(x',y'),
w=(x'',y'')$ element of $\mathrm{I\!R}^{2}$, we have
\begin{eqnarray*}
\delta^{2}f_{2}(u,v,w)&=& vf_{2}(u,w)-uf_{2}(v,w)+f_{2}(v,uw)-f_{2}(u,vw)+f_{2}(uv,w)-f_{2}(vu,w)\\
                      &+& f_{2}(u,v)w-f_{2}(v,u)w\\
                      &=& \mu(v,f_{2}(u,w))-\mu(u,f_{2}(v,w))+f_{2}(v,\mu(u,w))-f_{2}(u,\mu(v,w))+f_{2}(\mu(u,v),w)\\
                      &-& f_{2}(\mu(v,u),w)+ \mu(f_{2}(u,v),w)-\mu(f_{2}(v,u),w)
\end{eqnarray*}
where $f_{2}$ is a bilinear application  of $\mathrm{I\! R}^{2}
\times \mathrm{I\!R}^{2}$ in $\mathrm{I\!R}^{2}$, whose associated
with matrix:
\begin{eqnarray*}
     C= \begin{pmatrix}
          e&f\\
           g&h
            \end{pmatrix},
       D= \begin{pmatrix}
          i&j\\
           k&l
            \end{pmatrix}.
\end{eqnarray*}
So,  $f_{2} (u,v)=(u^{t}Cv, u^{t}Dv)= (exx'+fxy'+gx'y+hyy',
ixx'+jxy'+kx'y+lyy')$; we have\\
 $\bullet$
$vf_{2}(u,w)=(v^{t}\Gamma_{1}f_{2}(u,w),v^{t}\Gamma_{2}f_{2}(u,w)),$
with
\begin{eqnarray*}
 v^{t}\Gamma_{1}f_{2}(u,w)&=&(e\Gamma_{11}^{1}+i\Gamma_{12}^{1})xx'x''+(f\Gamma_{11}^{1}+j\Gamma_{12}^{1})xx'y''+(g\Gamma_{11}^{1}+k\Gamma_{12}^{1})x'x''y\\
                          &+&(h\Gamma_{11}^{1}+l\Gamma_{12}^{1})x'yy''+(e\Gamma_{21}^{1}+i\Gamma_{22}^{1})xx''y'+(f\Gamma_{21}^{1}+j\Gamma_{22}^{1})xy'y''\\
                          &+&(g\Gamma_{21}^{1}+k\Gamma_{22}^{1})x''yy'+(h\Gamma_{21}^{1}+l\Gamma_{22}^{1})yy'y'',
\end{eqnarray*}
\begin{eqnarray*}
v^{t}\Gamma_{2}f_{2}(u,w)&=&(e\Gamma_{11}^{2}+i\Gamma_{12}^{2})xx'x''+(f\Gamma_{11}^{2}+j\Gamma_{12}^{2})xx'y''+(g\Gamma_{11}^{2}+k\Gamma_{12}^{2})x'x''y\\
                         &+&(h\Gamma_{11}^{2}+l\Gamma_{12}^{2})x'yy''+(e\Gamma_{21}^{2}+i\Gamma_{22}^{2})xx''y'+(f\Gamma_{21}^{2}+j\Gamma_{22}^{2})xy'y''\\
                         &+&(g\Gamma_{}^{1}+k\Gamma_{22}^{2})x''yy'+(h\Gamma_{21}^{2}+l\Gamma_{22}^{2})yy'y''
\end{eqnarray*}
Then \begin{eqnarray*}vf_{2}(u,w) =
\begin{cases}\biggl((e\Gamma_{11}^{1}+i\Gamma_{12}^{1})xx'x''+(f\Gamma_{11}^{1}+j\Gamma_{12}^{1})xx'y''+(g\Gamma_{11}^{1}+k\Gamma_{12}^{1})x'x''y\\
+(h\Gamma_{11}^{1}+l\Gamma_{12}^{1})x'yy''+(e\Gamma_{21}^{1}+i\Gamma_{22}^{1})xx''y'+(f\Gamma_{21}^{1}+j\Gamma_{22}^{1})xy'y''\\
+(g\Gamma_{21}^{1}+k\Gamma_{22}^{1})x''yy'+(h\Gamma_{21}^{1}+l\Gamma_{22}^{1})yy'y'',\\
(e\Gamma_{11}^{2}+i\Gamma_{12}^{2})xx'x''+(f\Gamma_{11}^{2}+j\Gamma_{12}^{2})xx'y''+(g\Gamma_{11}^{2}+k\Gamma_{12}^{2})x'x''y\\
+(h\Gamma_{11}^{2}+l\Gamma_{12}^{2})x'yy''+(e\Gamma_{21}^{2}+i\Gamma_{22}^{2})xx''y'+(f\Gamma_{21}^{2}+j\Gamma_{22}^{2})xy'y''\\
+(g\Gamma_{}^{1}+k\Gamma_{22}^{2})x''yy'+(h\Gamma_{21}^{2}+l\Gamma_{22}^{2})yy'y'' \biggr)\\
\end{cases}
\end{eqnarray*}\\
    $\bullet$
$uf_{2}(v,w)=(u^{t}\Gamma_{1}f_{2}(v,w),u^{t}\Gamma_{2}f_{2}(v,w))$,
with
\begin{eqnarray*}
u^{t}\Gamma_{1}f_{2}(v,w)&=&(e\Gamma_{11}^{1}+i\Gamma_{12}^{1})xx'x''+(f\Gamma_{11}^{1}+j\Gamma_{12}^{1})xx'y''+(g\Gamma_{11}^{1}+k\Gamma_{12}^{1})xx''y'\\
                         &+&(h\Gamma_{11}^{1}+l\Gamma_{12}^{1})xy'y''+(e\Gamma_{21}^{1}+i\Gamma_{22}^{1})x'x''y+(f\Gamma_{21}^{1}+j\Gamma_{22}^{1})x'yy''\\
                         &+&(g\Gamma_{21}^{1}+k\Gamma_{22}^{1})x''yy'+(h\Gamma_{21}^{1}+l\Gamma_{22}^{1})yy'y'',
\end{eqnarray*}
\begin{eqnarray*}
u^{t}\Gamma_{2}f_{2}(v,w)&=&(e\Gamma_{11}^{2}+i\Gamma_{12}^{2})xx'x''+(f\Gamma_{11}^{2}+j\Gamma_{12}^{2})xx'y''+(g\Gamma_{11}^{2}+k\Gamma_{12}^{2})xx''y'\\
                         &+&(h\Gamma_{11}^{2}+l\Gamma_{12}^{2})xy'y''+(e\Gamma_{21}^{2}+i\Gamma_{22}^{2})x'x''y+(f\Gamma_{21}^{2}+j\Gamma_{22}^{2})x'yy''\\
                         &+&(g\Gamma_{21}^{2}+k\Gamma_{22}^{2})x''yy'+(h\Gamma_{21}^{2}+l\Gamma_{22}^{2})yy'y''
\end{eqnarray*}
Then \begin{eqnarray*}uf_{2}(v,w) =
\begin{cases}\biggl((e\Gamma_{11}^{1}+i\Gamma_{12}^{1})xx'x''+(f\Gamma_{11}^{1}+j\Gamma_{12}^{1})xx'y''+(g\Gamma_{11}^{1}+k\Gamma_{12}^{1})xx''y'\\
+(h\Gamma_{11}^{1}+l\Gamma_{12}^{1})xy'y''+(e\Gamma_{21}^{1}+i\Gamma_{22}^{1})x'x''y+(f\Gamma_{21}^{1}+j\Gamma_{22}^{1})x'yy''\\
+(g\Gamma_{21}^{1}+k\Gamma_{22}^{1})x''yy'+(h\Gamma_{21}^{1}+l\Gamma_{22}^{1})yy'y'',\\
(e\Gamma_{11}^{2}+i\Gamma_{12}^{2})xx'x''+(f\Gamma_{11}^{2}+j\Gamma_{12}^{2})xx'y''+(g\Gamma_{11}^{2}+k\Gamma_{12}^{2})xx''y'\\
+(h\Gamma_{11}^{2}+l\Gamma_{12}^{2})xy'y''+(e\Gamma_{21}^{2}+i\Gamma_{22}^{2})x'x''y+(f\Gamma_{21}^{2}+j\Gamma_{22}^{2})x'yy''\\
+(g\Gamma_{21}^{2}+k\Gamma_{22}^{2})x''yy'+(h\Gamma_{21}^{2}+l\Gamma_{22}^{2})yy'y'' \biggr)\\
\end{cases}
\end{eqnarray*}\\
$\bullet$ $f_{2}(v,uw)=(v^{t}C(uw),v^{t}D(uw))$, with
\begin{eqnarray*}
v^{t}C(uw) &=&(e\Gamma_{11}^{1}+f\Gamma_{11}^{2})xx'x''+(e\Gamma_{12}^{1}+f\Gamma_{12}^{2})xx'y''+(e\Gamma_{21}^{1}+f\Gamma_{21}^{2})x'x''y\\
           &+&(e\Gamma_{22}^{1}+f\Gamma_{22}^{2})x'yy''+(g\Gamma_{11}^{1}+h\Gamma_{11}^{2})xy'x''+(g\Gamma_{12}^{1}+h\Gamma_{12}^{2})xy'y''\\
           &+&(g\Gamma_{21}^{1}+h\Gamma_{21}^{2})x''yy'+(g\Gamma_{22}^{1}+l\Gamma_{22}^{2})yy'y'',
\end{eqnarray*}
\begin{eqnarray*}
v^{t}D(uw) &=&(i\Gamma_{11}^{1}+j\Gamma_{11}^{2})xx'x''+(i\Gamma_{12}^{1}+j\Gamma_{12}^{2})xx'y''+(i\Gamma_{21}^{1}+j\Gamma_{21}^{2})x'x''y\\
           &+&(i\Gamma_{22}^{1}+j\Gamma_{22}^{2})x'yy''+(k\Gamma_{11}^{1}+l\Gamma_{11}^{2})xy'x''+(k\Gamma_{12}^{1}+l\Gamma_{12}^{2})xy'y''\\
           &+&(k\Gamma_{21}^{1}+l\Gamma_{21}^{2})x''yy'+(k\Gamma_{22}^{1}+l\Gamma_{22}^{2})yy'y''
\end{eqnarray*}
Then \begin{eqnarray*}f_{2}(v,uw) =
\begin{cases}\biggl((e\Gamma_{11}^{1}+f\Gamma_{11}^{2})xx'x''+(e\Gamma_{12}^{1}+f\Gamma_{12}^{2})xx'y''+(e\Gamma_{21}^{1}+f\Gamma_{21}^{2})x'x''y\\
+(e\Gamma_{22}^{1}+f\Gamma_{22}^{2})x'yy''+(g\Gamma_{11}^{1}+h\Gamma_{11}^{2})xy'x''+(g\Gamma_{12}^{1}+h\Gamma_{12}^{2})xy'y''\\
+(g\Gamma_{21}^{1}+h\Gamma_{21}^{2})x''yy'+(g\Gamma_{22}^{1}+l\Gamma_{22}^{2})yy'y'',\\
(i\Gamma_{11}^{1}+j\Gamma_{11}^{2})xx'x''+(i\Gamma_{12}^{1}+j\Gamma_{12}^{2})xx'y''+(i\Gamma_{21}^{1}+j\Gamma_{21}^{2})x'x''y\\
+(i\Gamma_{22}^{1}+j\Gamma_{22}^{2})x'yy''+(k\Gamma_{11}^{1}+l\Gamma_{11}^{2})xy'x''+(k\Gamma_{12}^{1}+l\Gamma_{12}^{2})xy'y''\\
+(k\Gamma_{21}^{1}+l\Gamma_{21}^{2})x''yy'+(k\Gamma_{22}^{1}+l\Gamma_{22}^{2})yy'y'' \biggr)\\
\end{cases}
\end{eqnarray*}\\
$\bullet$ $f_{2}(u,vw)=(u^{t}C(vw),u^{t}D(vw))$, with
\begin{eqnarray*}
u^{t}C(vw) &=&(e\Gamma_{11}^{1}+f\Gamma_{11}^{2})xx'x''+(e\Gamma_{12}^{1}+f\Gamma_{12}^{2})xx'y''+(e\Gamma_{21}^{1}+f\Gamma_{21}^{2})xx''y'\\
           &+&(e\Gamma_{22}^{1}+f\Gamma_{22}^{2})xy'y''+(g\Gamma_{11}^{1}+h\Gamma_{11}^{2})x'yx''+(g\Gamma_{12}^{1}+h\Gamma_{12}^{2})x'yy''\\
           &+&(g\Gamma_{21}^{1}+h\Gamma_{21}^{2})x''yy'+(g\Gamma_{22}^{1}+h\Gamma_{22}^{2})yy'y'',
\end{eqnarray*}
\begin{eqnarray*}
u^{t}D(vw) &=&(i\Gamma_{11}^{1}+j\Gamma_{11}^{2})xx'x''+(i\Gamma_{12}^{1}+j\Gamma_{12}^{2})xx'y''+(i\Gamma_{21}^{1}+j\Gamma_{21}^{2})xx''y'\\
           &+&(i\Gamma_{22}^{1}+j\Gamma_{22}^{2})xy'y''+(k\Gamma_{11}^{1}+l\Gamma_{11}^{2})x'yx''+(k\Gamma_{12}^{1}+l\Gamma_{12}^{2})x'yy''\\
           &+&(k\Gamma_{21}^{1}+l\Gamma_{21}^{2})x''yy'+(k\Gamma_{22}^{1}+l\Gamma_{22}^{2})yy'y''
\end{eqnarray*}
Then \begin{eqnarray*}f_{2}(u,vw) =
\begin{cases}\biggl((e\Gamma_{11}^{1}+f\Gamma_{11}^{2})xx'x''+(e\Gamma_{12}^{1}+f\Gamma_{12}^{2})xx'y''+(e\Gamma_{21}^{1}+f\Gamma_{21}^{2})xx''y'\\
+(e\Gamma_{22}^{1}+f\Gamma_{22}^{2})xy'y''+(g\Gamma_{11}^{1}+h\Gamma_{11}^{2})x'yx''+(g\Gamma_{12}^{1}+h\Gamma_{12}^{2})x'yy''\\
+(g\Gamma_{21}^{1}+h\Gamma_{21}^{2})x''yy'+(g\Gamma_{22}^{1}+h\Gamma_{22}^{2})yy'y'',\\
(i\Gamma_{11}^{1}+j\Gamma_{11}^{2})xx'x''+(i\Gamma_{12}^{1}+j\Gamma_{12}^{2})xx'y''+(i\Gamma_{21}^{1}+j\Gamma_{21}^{2})xx''y'\\
+(i\Gamma_{22}^{1}+j\Gamma_{22}^{2})xy'y''+(k\Gamma_{11}^{1}+l\Gamma_{11}^{2})x'yx''+(k\Gamma_{12}^{1}+l\Gamma_{12}^{2})x'yy''\\
+(k\Gamma_{21}^{1}+l\Gamma_{21}^{2})x''yy'+(k\Gamma_{22}^{1}+l\Gamma_{22}^{2})yy'y'' \biggr)\\
\end{cases}
\end{eqnarray*}\\
$\bullet$ $f_{2}(uv,w)=((uv)^{t}Cw,(uv)^{t}Dw)$, with
\begin{eqnarray*}
(uv)^{t}Cw &=&(e\Gamma_{11}^{1}+f\Gamma_{11}^{2})xx'x''+(e\Gamma_{12}^{1}+f\Gamma_{12}^{2})xx''y'+(e\Gamma_{21}^{1}+f\Gamma_{21}^{2})x'x''y\\
           &+&(e\Gamma_{22}^{1}+f\Gamma_{22}^{2})x''yy'+(g\Gamma_{11}^{1}+h\Gamma_{11}^{2})xy''x'+(g\Gamma_{12}^{1}+h\Gamma_{12}^{2})xy'y''\\
           &+&(g\Gamma_{21}^{1}+h\Gamma_{21}^{2})x'yy''+(g\Gamma_{22}^{1}+h\Gamma_{22}^{2})yy'y'',
\end{eqnarray*}
\begin{eqnarray*}
(uv)^{t}Dw &=&(i\Gamma_{11}^{1}+j\Gamma_{11}^{2})xx'x''+(i\Gamma_{12}^{1}+j\Gamma_{12}^{2})xx''y'+(i\Gamma_{21}^{1}+j\Gamma_{21}^{2})x'x''y\\
           &+&(i\Gamma_{22}^{1}+j\Gamma_{22}^{2})x''yy'+(k\Gamma_{11}^{1}+l\Gamma_{11}^{2})xy''x'+(k\Gamma_{12}^{1}+l\Gamma_{12}^{2})xy'y''\\
           &+&(k\Gamma_{21}^{1}+l\Gamma_{21}^{2})x'yy''+(k\Gamma_{22}^{1}+l\Gamma_{22}^{2})yy'y''
\end{eqnarray*}
Then  \begin{eqnarray*}f_{2}(uv,w) =
\begin{cases}\biggl((e\Gamma_{11}^{1}+f\Gamma_{11}^{2})xx'x''+(e\Gamma_{12}^{1}+f\Gamma_{12}^{2})xx''y'+(e\Gamma_{21}^{1}+f\Gamma_{21}^{2})x'x''y\\
+(e\Gamma_{22}^{1}+f\Gamma_{22}^{2})x''yy'+(g\Gamma_{11}^{1}+h\Gamma_{11}^{2})xy''x'+(g\Gamma_{12}^{1}+h\Gamma_{12}^{2})xy'y''\\
+(g\Gamma_{21}^{1}+h\Gamma_{21}^{2})x'yy''+(g\Gamma_{22}^{1}+h\Gamma_{22}^{2})yy'y'',\\
(i\Gamma_{11}^{1}+j\Gamma_{11}^{2})xx'x''+(i\Gamma_{12}^{1}+j\Gamma_{12}^{2})xx''y'+(i\Gamma_{21}^{1}+j\Gamma_{21}^{2})x'x''y\\
+(i\Gamma_{22}^{1}+j\Gamma_{22}^{2})x''yy'+(k\Gamma_{11}^{1}+l\Gamma_{11}^{2})xy''x'+(k\Gamma_{12}^{1}+l\Gamma_{12}^{2})xy'y''\\
+(k\Gamma_{21}^{1}+l\Gamma_{21}^{2})x'yy''+(k\Gamma_{22}^{1}+l\Gamma_{22}^{2})yy'y'' \biggr)\\
\end{cases}
\end{eqnarray*}\\
$\bullet$ $f_{2}(vu,w)=((vu)^{t}Cw,(vu)^{t}Dw)$, with
\begin{eqnarray*}
(vu)^{t}Cw &=&(e\Gamma_{11}^{1}+g\Gamma_{11}^{2})xx'x''+(e\Gamma_{12}^{1}+g\Gamma_{12}^{2})x'x''y+(e\Gamma_{21}^{1}+g\Gamma_{21}^{2})xx''y'\\
           &+&(e\Gamma_{22}^{1}+g\Gamma_{22}^{2})x''yy'+(f\Gamma_{11}^{1}+h\Gamma_{11}^{2})xy''x'+(f\Gamma_{12}^{1}+h\Gamma_{12}^{2})x'yy''\\
           &+&(f\Gamma_{21}^{1}+h\Gamma_{21}^{2})xy'y''+(f\Gamma_{22}^{1}+h\Gamma_{22}^{2})yy'y'',
\end{eqnarray*}
\begin{eqnarray*}
(vu)^{t}Dw &=&(i\Gamma_{11}^{1}+k\Gamma_{11}^{2})xx'x''+(i\Gamma_{12}^{1}+k\Gamma_{12}^{2})x'x''y+(i\Gamma_{21}^{1}+k\Gamma_{21}^{2})xx''y'\\
           &+&(i\Gamma_{22}^{1}+k\Gamma_{22}^{2})x''yy'+(j\Gamma_{11}^{1}+l\Gamma_{11}^{2})xy''x'+(j\Gamma_{12}^{1}+k\Gamma_{12}^{2})x'yy''\\
           &+&(j\Gamma_{21}^{1}+k\Gamma_{21}^{2})xy'y''+(j\Gamma_{22}^{1}+k\Gamma_{22}^{2})yy'y''
\end{eqnarray*}
Then  \begin{eqnarray*}f_{2}(vu,w) =
\begin{cases}\biggl((e\Gamma_{11}^{1}+g\Gamma_{11}^{2})xx'x''+(e\Gamma_{12}^{1}+g\Gamma_{12}^{2})x'x''y+(e\Gamma_{21}^{1}+g\Gamma_{21}^{2})xx''y'\\
+(e\Gamma_{22}^{1}+g\Gamma_{22}^{2})x''yy'+(f\Gamma_{11}^{1}+h\Gamma_{11}^{2})xy''x'+(f\Gamma_{12}^{1}+h\Gamma_{12}^{2})x'yy''\\
+(f\Gamma_{21}^{1}+h\Gamma_{21}^{2})xy'y''+(f\Gamma_{22}^{1}+h\Gamma_{22}^{2})yy'y'',\\
(i\Gamma_{11}^{1}+k\Gamma_{11}^{2})xx'x''+(i\Gamma_{12}^{1}+k\Gamma_{12}^{2})x'x''y+(i\Gamma_{21}^{1}+k\Gamma_{21}^{2})xx''y'\\
+(i\Gamma_{22}^{1}+k\Gamma_{22}^{2})x''yy'+(j\Gamma_{11}^{1}+l\Gamma_{11}^{2})xy''x'+(j\Gamma_{12}^{1}+k\Gamma_{12}^{2})x'yy''\\
+(j\Gamma_{21}^{1}+k\Gamma_{21}^{2})xy'y''+(j\Gamma_{22}^{1}+k\Gamma_{22}^{2})yy'y''\biggr)\\
\end{cases}
\end{eqnarray*}
$\bullet$
$f_{2}(u,v)w=(f_{2}(u,v)^{t}\Gamma_{1}w,f_{2}(u,v)^{t}\Gamma_{2}w)$,
with
\begin{eqnarray*}
f_{2}(u,v)^{t}\Gamma_{1}w&=&(e\Gamma_{11}^{1}+i\Gamma_{21}^{1})xx'x''+(f\Gamma_{11}^{1}+j\Gamma_{21}^{1})xx''y'+(g\Gamma_{11}^{1}+k\Gamma_{21}^{1})x'x''y\\
                 &+&(h\Gamma_{11}^{1}+l\Gamma_{21}^{1})x''yy'+(e\Gamma_{12}^{1}+i\Gamma_{22}^{1})xy''x'+(f\Gamma_{12}^{1}+j\Gamma_{22}^{1})xy'y''\\
                 &+&(g\Gamma_{12}^{1}+k\Gamma_{22}^{1})x'yy''+(h\Gamma_{12}^{1}+l\Gamma_{22}^{1})yy'y'',
\end{eqnarray*}
\begin{eqnarray*}
f_{2}(u,v)^{t}\Gamma_{2}w&=&(e\Gamma_{11}^{2}+i\Gamma_{21}^{2})xx'x''+(f\Gamma_{11}^{2}+j\Gamma_{21}^{2})xx''y'+(g\Gamma_{11}^{2}+k\Gamma_{21}^{2})x'x''y\\
                 &+&(h\Gamma_{11}^{2}+l\Gamma_{21}^{2})x''yy'+(e\Gamma_{12}^{2}+i\Gamma_{22}^{2})xy''x'+(f\Gamma_{12}^{2}+J\Gamma_{22}^{2})xy'y''\\
                 &+&(g\Gamma_{12}^{2}+k\Gamma_{22}^{2})x'yy''+(h\Gamma_{12}^{2}+l\Gamma_{22}^{2})yy'y''
\end{eqnarray*}
Then \begin{eqnarray*}f_{2}(u,v)w =
\begin{cases}\biggl((e\Gamma_{11}^{1}+i\Gamma_{21}^{1})xx'x''+(f\Gamma_{11}^{1}+j\Gamma_{21}^{1})xx''y'+(g\Gamma_{11}^{1}+k\Gamma_{21}^{1})x'x''y\\
+(h\Gamma_{11}^{1}+l\Gamma_{21}^{1})x''yy'+(e\Gamma_{12}^{1}+i\Gamma_{22}^{1})xy''x'+(f\Gamma_{12}^{1}+j\Gamma_{22}^{1})xy'y''\\
+(g\Gamma_{12}^{1}+k\Gamma_{22}^{1})x'yy''+(h\Gamma_{12}^{1}+l\Gamma_{22}^{1})yy'y'',\\
(e\Gamma_{11}^{2}+i\Gamma_{21}^{2})xx'x''+(f\Gamma_{11}^{2}+j\Gamma_{21}^{2})xx''y'+(g\Gamma_{11}^{2}+k\Gamma_{21}^{2})x'x''y\\
+(h\Gamma_{11}^{2}+l\Gamma_{21}^{2})x''yy'+(e\Gamma_{12}^{2}+i\Gamma_{22}^{2})xy''x'+(f\Gamma_{12}^{2}+J\Gamma_{22}^{2})xy'y''\\
+(g\Gamma_{12}^{2}+k\Gamma_{22}^{2})x'yy''+(h\Gamma_{12}^{2}+l\Gamma_{22}^{2})yy'y'' \biggr)\\
\end{cases}
\end{eqnarray*}\\
$\bullet$
$f_{2}(v,u)w=(f_{2}(v,u)^{t}\Gamma_{1}w,f_{2}(v,u)^{t}\Gamma_{2}w)$,
with,
\begin{eqnarray*}
f_{2}(v,u)^{t}\Gamma_{1}w&=&(e\Gamma_{11}^{1}+i\Gamma_{21}^{1})xx'x''+(f\Gamma_{11}^{1}+j\Gamma_{21}^{1})x'x''y+(g\Gamma_{11}^{1}+k\Gamma_{21}^{1})xx''y'\\
                         &+&(h\Gamma_{11}^{1}+l\Gamma_{21}^{1})x''yy'+(e\Gamma_{12}^{1}+i\Gamma_{22}^{1})xy''x'+(f\Gamma_{12}^{1}+j\Gamma_{22}^{1})x'yy''\\
                         &+&(g\Gamma_{12}^{1}+k\Gamma_{22}^{1})xy'y''+(h\Gamma_{12}^{1}+l\Gamma_{22}^{1})yy'y'',
\end{eqnarray*}
\begin{eqnarray*}
f_{2}(v,u)^{t}\Gamma_{2}w&=&(e\Gamma_{11}^{2}+i\Gamma_{21}^{2})xx'x''+(f\Gamma_{11}^{2}+j\Gamma_{21}^{2})x'x''y+(g\Gamma_{11}^{2}+k\Gamma_{21}^{2})xx''y'\\
                         &+&(h\Gamma_{11}^{2}+l\Gamma_{21}^{2})x''yy'+(e\Gamma_{12}^{2}+i\Gamma_{22}^{2})xy''x'+(f\Gamma_{12}^{2}+j\Gamma_{22}^{2})x'yy''\\
                         &+&(g\Gamma_{12}^{2}+k\Gamma_{22}^{2})xy'y''+(h\Gamma_{12}^{2}+l\Gamma_{22}^{2})yy'y''
\end{eqnarray*}
Then  \begin{eqnarray*}f_{2}(v,u)w =
\begin{cases}\biggl((e\Gamma_{11}^{1}+i\Gamma_{21}^{1})xx'x''+(f\Gamma_{11}^{1}+j\Gamma_{21}^{1})x'x''y+(g\Gamma_{11}^{1}+k\Gamma_{21}^{1})xx''y'\\
+(h\Gamma_{11}^{1}+l\Gamma_{21}^{1})x''yy'+(e\Gamma_{12}^{1}+i\Gamma_{22}^{1})xy''x'+(f\Gamma_{12}^{1}+j\Gamma_{22}^{1})x'yy''\\
+(g\Gamma_{12}^{1}+k\Gamma_{22}^{1})xy'y''+(h\Gamma_{12}^{1}+l\Gamma_{22}^{1})yy'y'',\\
(e\Gamma_{11}^{2}+i\Gamma_{21}^{2})xx'x''+(f\Gamma_{11}^{2}+j\Gamma_{21}^{2})x'x''y+(g\Gamma_{11}^{2}+k\Gamma_{21}^{2})xx''y'\\
+(h\Gamma_{11}^{2}+l\Gamma_{21}^{2})x''yy'+(e\Gamma_{12}^{2}+i\Gamma_{22}^{2})xy''x'+(f\Gamma_{12}^{2}+j\Gamma_{22}^{2})x'yy''\\
+(g\Gamma_{12}^{2}+k\Gamma_{22}^{2})xy'y''+(h\Gamma_{12}^{2}+l\Gamma_{22}^{2})yy'y'' \biggr).\\
\end{cases}
\end{eqnarray*}
We have the expression $\delta^{2}f_{2}(u,v,w)$:\\
\begin{eqnarray}\delta^{2}f_{2}(u,v,w)=
\begin{cases}\biggl(((\Gamma_{12}^{1}-\Gamma_{21}^{1})e+\Gamma_{22}^{1}i+(\Gamma_{11}^{1}-\Gamma_{21}^{2})f+\Gamma_{21}^{1}j+(\Gamma_{12}^{2}\\
-\Gamma_{11}^{1}-\Gamma_{21}^{2})g-(\Gamma_{12}^{1}+\Gamma_{21}^{1})k+\Gamma_{11}^{2}h)x'x''y\\
-(((\Gamma_{12}^{1}-\Gamma_{21}^{1})e+\Gamma_{22}^{1}i+(\Gamma_{11}^{1}-\Gamma_{21}^{2})f+\Gamma_{21}^{1}j+(\Gamma_{12}^{2}-\Gamma_{11}^{1}\\
-\Gamma_{21}^{2})g-(\Gamma_{12}^{1}+\Gamma_{21}^{1})k+\Gamma_{11}^{2}h)xx''y'\\
+ (\Gamma_{21}^{2}e+(\Gamma_{12}^{1}-2\Gamma_{21}^{1}+\Gamma_{22}^{2})i+\Gamma_{11}^{2}f-2\Gamma_{11}^{2}g+(\Gamma_{11}^{1}-2\Gamma_{21}^{2})k\\
+\Gamma_{11}^{2}l)x'yy''\\
-(\Gamma_{21}^{2}e+(\Gamma_{12}^{1}-2\Gamma_{21}^{1}+\Gamma_{22}^{2})i+\Gamma_{11}^{2}f-2\Gamma_{11}^{2}g+(\Gamma_{11}^{1}-2\Gamma_{21}^{2})k\\
+\Gamma_{11}^{2}l)xy'y'',\\
(-\Gamma_{22}^{1}e+(2\Gamma_{12}^{1}-\Gamma_{22}^{2})f+2\Gamma_{22}^{1}j-\Gamma_{22}^{1}k+(2\Gamma_{12}^{2}-\Gamma_{11}^{1}-\Gamma_{21}^{2})h\\
-\Gamma_{12}^{1}l)x'x''y\\
-(-\Gamma_{22}^{1}e+(2\Gamma_{12}^{1}-\Gamma_{22}^{2})f+2\Gamma_{22}^{1}j-\Gamma_{22}^{1}k+(2\Gamma_{12}^{2}-\Gamma_{11}^{1}-\Gamma_{21}^{2})h\\
-\Gamma_{12}^{1}l)xx''y'\\
+(-\Gamma_{22}^{1}i+(\Gamma_{12}^{2}+\Gamma_{21}^{2})f+(\Gamma_{12}^{1}-\Gamma_{21}^{1}+\Gamma_{22}^{2})j-\Gamma_{12}^{1}g\\
+(\Gamma_{12}^{1}-\Gamma_{22}^{2})k-\Gamma_{11}^{2}h+(\Gamma_{12}^{2}-\Gamma_{21}^{2})l)x'yy''\\
-(-\Gamma_{22}^{1}i+(\Gamma_{12}^{2}+\Gamma_{21}^{2})f+(\Gamma_{12}^{1}-\Gamma_{21}^{1}+\Gamma_{22}^{2})j-\Gamma_{12}^{1}g\\
+(\Gamma_{12}^{1}-\Gamma_{22}^{2})k-\Gamma_{11}^{2}h+(\Gamma_{12}^{2}-\Gamma_{21}^{2})l)xy'y'')\biggr)\\
\end{cases}
\end{eqnarray}
So, we  have
\begin{eqnarray*}
f_{2}\in Ker\delta^{2} &\Longleftrightarrow & \delta^{2}f_{2}(u,v,w)=0\\
                       &\Longleftrightarrow & \delta^{2}f_{2}(e_{i},e_{j},e_{k})=0, 1\leq i, j, k,\leq 2
\end{eqnarray*}
equal to
\begin{eqnarray}
\begin{cases}
((\Gamma_{12}^{1}-\Gamma_{21}^{1})e+\Gamma_{22}^{1}i+(\Gamma_{11}^{1}-\Gamma_{21}^{2})f+\Gamma_{21}^{1}j+(\Gamma_{12}^{2}-\Gamma_{11}^{1}
-\Gamma_{21}^{2})g-(\Gamma_{12}^{1}+\Gamma_{21}^{1})k+\Gamma_{11}^{2}h=0\\
\Gamma_{21}^{2}e+(\Gamma_{12}^{1}-2\Gamma_{21}^{1}+\Gamma_{22}^{2})i+\Gamma_{11}^{2}f-2\Gamma_{11}^{2}g+(\Gamma_{11}^{1}-2\Gamma_{21}^{2})k
+\Gamma_{11}^{2}l)=0\\
-\Gamma_{22}^{1}e+(2\Gamma_{12}^{1}-\Gamma_{22}^{2})f+2\Gamma_{22}^{1}j-\Gamma_{22}^{1}k+(2\Gamma_{12}^{2}-\Gamma_{11}^{1}-\Gamma_{21}^{2})h
-\Gamma_{12}^{1}l=0\\
-\Gamma_{22}^{1}i+(\Gamma_{12}^{2}+\Gamma_{21}^{2})f+(\Gamma_{12}^{1}-\Gamma_{21}^{1}+\Gamma_{22}^{2})j-\Gamma_{12}^{1}g+(\Gamma_{12}^{1}-\Gamma_{22}^{2})k
-\Gamma_{11}^{2}h+(\Gamma_{12}^{2}-\Gamma_{21}^{2})l=0\\
\end{cases}
\end{eqnarray}
\end{proof}

The following lemma characterizes $Im\delta^{q-1}, q= 0, 1,2$,  on
$\mathrm{I\!R}^{2}$.
\begin{theorem}
For the KV-algebra $(\mathrm{I\!R}^{2}, \mu)$, we have:\\
(i)$Im\delta^{-1}=\{0\}$ \\
(ii)$Im\delta^{0}$ is the set of linear applications $g\approx
\begin{pmatrix}
          u_{11}&u_{12}\\
          u_{21}&u_{22}
            \end{pmatrix}$ of $\mathrm{I\!R}^{2}$ satisfying the following equation Eq.~(\ref{eq5})
\begin{eqnarray}
\begin{cases}\label{eq5}
(\Gamma_{21}^{1}-\Gamma_{12}^{1})\xi_{2}=u_{11}\\
(\Gamma_{12}^{1}-\Gamma_{21}^{1})\xi_{1}=u_{12}\\
(\Gamma_{21}^{2}-\Gamma_{12}^{2})\xi_{2}=u_{21}\\
(\Gamma_{12}^{2}-\Gamma_{21}^{2})\xi_{1}=u_{22}\\
\end{cases}
\end{eqnarray}
(iii) $Im\delta^{1}$ is the set of bilinear applications $g\approx
\biggl(\begin{pmatrix}
          u_{11}&u_{12}\\
          u_{21}&u_{22}
            \end{pmatrix},
        \begin{pmatrix}
         v_{11}&v_{12}\\
         v_{221}&v_{22}
            \end{pmatrix}\biggr)$ of $\mathrm{I\!R}^{2}$ verifying the following Eq.~(\ref{eq6}) equation:
\begin{eqnarray}
\begin{cases}\label{eq6}
-\Gamma_{11}^{1}\alpha+\Gamma_{11}^{2}\beta-(\Gamma_{12}^{1}+\Gamma_{21}^{1})\gamma=u_{11}\\
(\Gamma_{12}^{2}-\Gamma_{11}^{1})\beta-\Gamma_{22}^{1}\gamma-\Gamma_{12}^{1}\lambda=u_{12}\\
(\Gamma_{11}^{2}-\Gamma_{11}^{1})\beta-\Gamma_{22}^{1}\gamma-\Gamma_{21}^{1}\lambda=u_{21}\\
\Gamma_{22}^{1}\alpha+(\Gamma_{22}^{2}-\Gamma_{12}^{1}-\Gamma_{21}^{1})\beta-2\Gamma_{22}^{1}\lambda=u_{22}\\
-2\Gamma_{11}^{2}\alpha+(\Gamma_{11}^{1}-\Gamma_{12}^{2}-\Gamma_{21}^{2})\gamma+\Gamma_{11}^{2}\lambda=v_{11}\\
-\Gamma_{12}^{2}\alpha-\Gamma_{11}^{2}\beta+(\Gamma_{12}^{1}-\Gamma_{22}^{2})\gamma=v_{12}\\
-\Gamma_{21}^{2}\alpha-\Gamma_{11}^{2}\beta+(\Gamma_{21}^{1}-\Gamma_{22}^{2})\gamma=v_{21}\\
-(\Gamma_{21}^{2}+\Gamma_{12}^{2})\beta+\Gamma_{22}^{1}\gamma-\Gamma_{22}^{2}\lambda=v_{22}
\end{cases}
\end{eqnarray}
\end{theorem}
\begin{proof}
(i) According to KV-complex, we have  $Im\delta^{-1}=\{0\}$. \\
(ii) Let the linear application $g$ of $\mathrm{I\! R}^{2}$ in
$\mathrm{I\! R}^{2}$ the matrix
    $U =\begin{pmatrix}
         u_{11}&u_{12}\\
         u_{21}&u_{22}
            \end{pmatrix}$.\\
For all $u=(x,y) \in \mathrm{I\! R}^{2}$, we have
$g(u)=(u_{11}x+u_{12}y,u_{21}x+u_{22}y)$ . Then,
\begin{eqnarray*}
g\in Im\delta^{0} &\Longleftrightarrow & \delta^{0}(\xi)(u)=g(u),\xi \in J(\mathbb{R}^{2})\\
                  &\Longleftrightarrow &\delta^{0}(\xi)(e_{i})=g(e_{i}), 1\leq i \leq 2
\end{eqnarray*}
equal to
\begin{eqnarray}
\begin{cases}
(\Gamma_{21}^{1}-\Gamma_{12}^{1})\xi_{2}=u_{11}\\
(\Gamma_{12}^{1}-\Gamma_{21}^{1})\xi_{1}=u_{12}\\
(\Gamma_{21}^{2}-\Gamma_{12}^{2})\xi_{2}=u_{21}\\
(\Gamma_{12}^{2}-\Gamma_{21}^{2})\xi_{1}=u_{22}\\
\end{cases}
\end{eqnarray}
(iii) Let the bilinear application $g$ of $\mathrm{I\!R}^{2} \times
\mathrm{I\! R}^{2}$ dans $\mathrm{I\!R}^{2}$ of matrix associated
\begin{eqnarray*}
     U= \begin{pmatrix}
          u_{11}& u_{12}\\
           u_{21}& u_{22}
            \end{pmatrix},
       V= \begin{pmatrix}
          v_{11}& v_{12}\\
           v_{21}& v_{22}
            \end{pmatrix},
\end{eqnarray*}
in the basis $(e_{1}, e_{2})$ de $\mathrm{I\! R}^{2} $.\\
For any  $u=(x,y),v=(x',y')$ elements of $\mathrm{I\! R}^{2}$, we have\\
 $g(u,v)=(u_{11} xx'+ u_{12}x y'+u_{21}x'y+u_{22}yy',v_{11} xx'+ v_{12}x y'+v_{21}x'y+v_{22}yy')$ we have\\
\begin{eqnarray*}
g\in Im \delta^{1} &\Longleftrightarrow & \delta^{1}f_{1}(u,v)=g(u,v)\\
                   &\Longleftrightarrow & \delta^{1}f_{1}(e_{i},e_{j})=g(e_{i},e_{j}), 1\leq i, j \leq 2\\
\end{eqnarray*}
equal to
\begin{eqnarray}
\begin{cases}
-\Gamma_{11}^{1}\alpha+\Gamma_{11}^{2}\beta-(\Gamma_{12}^{1}+\Gamma_{21}^{1})\gamma=u_{11}\\
(\Gamma_{12}^{2}-\Gamma_{11}^{1})\beta-\Gamma_{22}^{1}\gamma-\Gamma_{12}^{1}\lambda=u_{12}\\
(\Gamma_{11}^{2}-\Gamma_{11}^{1})\beta-\Gamma_{22}^{1}\gamma-\Gamma_{21}^{1}\lambda=u_{21}\\
\Gamma_{22}^{1}\alpha+(\Gamma_{22}^{2}-\Gamma_{12}^{1}-\Gamma_{21}^{1})\beta-2\Gamma_{22}^{1}\lambda=u_{22}\\
-2\Gamma_{11}^{2}\alpha+(\Gamma_{11}^{1}-\Gamma_{12}^{2}-\Gamma_{21}^{2})\gamma+\Gamma_{11}^{2}\lambda=v_{11}\\
-\Gamma_{12}^{2}\alpha-\Gamma_{11}^{2}\beta+(\Gamma_{12}^{1}-\Gamma_{22}^{2})\gamma=v_{12}\\
-\Gamma_{21}^{2}\alpha-\Gamma_{11}^{2}\beta+(\Gamma_{21}^{1}-\Gamma_{22}^{2})\gamma=v_{21}\\
-(\Gamma_{21}^{2}+\Gamma_{12}^{2})\beta+\Gamma_{22}^{1}\gamma-\Gamma_{22}^{2}\lambda=v_{22}
\end{cases}
\end{eqnarray}
\end{proof}
\subsubsection{KV-Hessian cohomology on  $\mathbb{R}^{2}$}
The following proposition defines the spaces $H_{KV}^{0}(\mu),
H_{KV}^{1}(\mu)$ of a Hessian KV-structure on $\mathrm{I\!R}^{2}$.
\begin{theorem}
Let $\mu \in Sol(\mathrm{I\!R}^{2}, KV)$, be a KV-Hessian structure
defined by:
\begin{center}
  $\mu =\biggl(\begin{pmatrix}
          0&a\\
          a&0
            \end{pmatrix},
        \begin{pmatrix}
         b&0\\
         0&a
            \end{pmatrix}\biggr),   a\neq b \neq 0.$
\end{center}
We have
\begin{flushleft}
(i)$ H_{KV}^{0}(\mu)\simeq \mathrm{I\! R}^{2}$\\
(ii) $H_{KV}^{1}(\mu)\simeq \{0\}$
\end{flushleft}
\end{theorem}
\begin{proof}
For all $u=(x,y), v=(x',y')$; we have: $\mu(u,v) =(axy'+ayx', bxx'+ayy')$.\\
\textbf{- Case of $H_{KV}^{0}(\mu)$}\\
 $\mu$ is symmetrical, so we have: $\Gamma_{21}^{1}= \Gamma_{12}^{1}, \Gamma_{21}^{2}= \Gamma_{12}^{2}$.
 From the equation Eq.~(\ref{eq2}), we deduce that: $\xi=(\xi_{1},\xi_{2})$; and therefore $Ker\delta^{0}=Vect \biggl\{e_{1},e_{2} \biggr\}$.
  Hence:  $ H_{KV}^{0}(\mu)\cong \mathrm{I\! R}^{2}$.\\
\textbf{- Case of $H_{KV}^{1}(\mu)$}\\
From the equation Eq.~(\ref{eq3}), We have: $\lambda= \beta= \alpha=
\gamma = 0$. This implies that:
    $f_{1} =0$; and therefore: $Ker\delta^{1}=\{0\}$.\\
On the other hand, from the equation  Eq.~(\ref{eq4}) ,We have:
$u_{12}=u_{21}=u_{11}=u_{22}=0$; $g=0$; and therefore
$Im\delta^{0}=\{0\}$. Hence:  $H_{KV}^{1}(\mu)\cong  \biggl\{0\biggr
\}$.
\end{proof}
The following proposition defines the space $H_{KV}^{2}(\mu)$ when
$\mu$ is a Hessian KV-structure on $\mathrm{I\!R}^{2}$.
\begin{theorem}
Let $\mu \in Sol(\mathrm{I\!R}^{2}, KV)$ be a KV-Hessian structure
defined by
\begin{center}
 $\mu =\biggl(\begin{pmatrix}
          0&a\\
          a&0
            \end{pmatrix},
        \begin{pmatrix}
         b&0\\
         0&a
            \end{pmatrix}\biggr),   a\neq b \neq 0$.
\end{center}
 We have $H_{KV}^{2}(\mu)\cong \biggl\{(E_{12},0),(E_{21},0),(0,E_{22})\biggr\}.$
\end{theorem}
\begin{proof}
from the equation Eq.~(\ref{eq5}),  we have: $\begin{cases}
aj-2ak+bh=0\\
bf-2bg+bl=0\\
af-al=0\\
aj-ag-bh=0\\
     \end{cases}$.
This implies that:\\
$l=g=f, j=k+\frac{1}{2}f$. And therefore:
    $f_{2} =\biggl(\begin{pmatrix}
         e&f\\
         f&h
            \end{pmatrix},
        \begin{pmatrix}
          i&k+\frac{1}{2}f\\
         k+\frac{1}{2}f&f
            \end{pmatrix}\biggr)$. Hence:
\begin{center}
 $ Ker\delta^{2}=Vect \biggl\{(E_{11},0),(E_{12},0),(E_{21},0),(E_{22},0),(0,E_{11}),(0,E_{12}),(0,E_{21}),(0,E_{22})\biggr\}$.
\end{center}
On the other hand, from the equation  Eq.~(\ref{eq6}) we have: $
\begin{cases}
b\beta-2a\gamma=u_{11}\\
-a\lambda=u_{12}\\
b\beta-a\lambda=u_{21}\\
-a\beta=u_{22}\\
-2b\alpha+b\lambda=v_{11}\\
-b\beta=v_{12}\\
-b\beta=v_{21}\\
-a\lambda=v_{22}\\
     \end{cases}$. We deduce that:\\
 $v_{22}=u_{12}, v_{21}=v_{12}=\frac{b}{a} u_{22}$; and therefore
  $g =\biggl(\begin{pmatrix}
         u_{11}&u_{12}\\
         u_{21} & u_{22}
            \end{pmatrix},
        \begin{pmatrix}
        v_{11}&\frac{b}{a} u_{22}\\
        \frac{b}{a} u_{22}&u_{12}
            \end{pmatrix}\biggr)$. But\\
\begin{eqnarray*}
\delta^{2}g(u,v, w)=0 & \Longleftrightarrow & \delta^{2}g(e_{i},e_{j}, e_{k})=0, 1 \leq i, j, k \leq 2\\
& \Longleftrightarrow & \begin{cases}
u_{12}=0\\
u_{21}=u_{12}\\
     \end{cases}
\end{eqnarray*}
We have: $g =\biggl(\begin{pmatrix}
         u_{11}&0\\
         0 & u_{22}
            \end{pmatrix},
        \begin{pmatrix}
        v_{11}&\frac{b}{a} u_{22}\\
        \frac{b}{a} u_{22}&0
            \end{pmatrix}\biggr)$; and therefore:\\
 $ Im\delta^{1}=Vect \biggl\{(E_{11},0),(E_{22},0),(0,E_{11}),(0,E_{12}),(0,E_{21})\biggr\}$. Hence:
  $H_{KV}^{2}(\mu)\cong \biggl\{(E_{12},0),(E_{21},0),(0,E_{22})\biggr\}$.
\end{proof}
\section{Deformations of a Hessian KV-structure on $\mathbb{R}^{2}$}\label{sec4}
In this paragraph, we first justify that the deformations of a
KV-structure of a KV-algebra A are 2-cocycles. Furthermore, we
determine the deformations of a Hessian KV-structure on$\mathrm{I\!
R}^{2}$.
\subsection{Deformation of a KV-algebra}
$(A,\mu)$ is a KV-algebra. Let $\mu_{t}$ be a deformation of $\mu$;
we have: $\mu_{t}= \mu+t\nu_{1}+t^{2}\nu_{2}+t^{3}\nu_{3}+. . .$
\begin{proposition}
$\mu_{t}$ is a KV-structure $\Leftrightarrow$ $\nu_{i} \in
Ker\delta^{2}$, $i=1,2,3,\dots$
\end{proposition}
\begin{proof}
Let's assume $s=p$ fix. We define \textbf{-the associator of
$\mu_{t}$ by:}
$Ass_{\mu_{t}}(u,v, w)= \mu_{t}(\mu_{t}(u,v),w)-\mu_{_{t}}(u,\mu_{t}(v,w)$,\\
for all u, v, w elements of $\mathrm{I\! R}^{2}$; with:
\begin{eqnarray*}
\mu_{t}(\mu_{t}(u,v),w)&=& (\mu+\sum_{s=1}^{p}t^{s}\nu_{s})((\mu+\sum_{s=1}^{p}t^{s}\nu_{s})(u,v),w)\\
                       &=& \mu(\mu(u,v),w)+\sum_{s=1}^{p}t^{s}[\nu_{s}(\mu(u,v),w)+\mu(\nu_{s}(u,v),w)]
\end{eqnarray*}
and
\begin{eqnarray*}
\mu_{t}(u,\mu_{t}(v,w))&=& (\mu+\sum_{s=1}^{p}t^{s}\nu_{s})(u,(\mu+\sum_{s=1}^{p}t^{s}\nu_{s})(v,w))\\
                       &=& \mu(u,\mu(v,w))+\sum_{s=1}^{p}t^{s}[\nu_{s}(u,\mu(v,w))+\mu(u,\nu_{s}(v,w))]
\end{eqnarray*}
So
\begin{eqnarray}
Ass_{\mu_{t}}(u,v, w) &=& Ass_{\mu}(u,v, w)+\sum_{s=1}^{p}t^{s}[\nu_{s}(\mu(u,v),w)+\mu(\nu_{s}(u,v),w) \nonumber \\
                      &-&\nu_{s}(u,\mu(v,w))-\mu(u,\nu_{s}(v,w))]
\end{eqnarray}
\textbf{the KV-anomaly of $\mu_{t}$ by:}
$KV_{\mu_{t}}(u,v,w)=Ass_{\mu_{t}}(u,v, w)-Ass_{\mu_{t}}(v,u, w)$,
for all u, v and w elements of $\mathrm{I\! R}^{2}$; with:
\begin{eqnarray*}
Ass_{\mu_{t}}(v,u, w)&=& Ass_{\mu}(v,u, w)+\sum_{s=1}^{p}t^{s}[\nu_{s}(\mu(v,u),w)+\mu(\nu_{s}(v,u),w)\\
                     &-& \nu_{s}(v,\mu(u,w))-\mu(v,\nu_{s}(u,w))]
\end{eqnarray*}
So:
\begin{eqnarray*}
KV_{\mu_{t}}(u, v,w) &=& KV_{\mu}(u, v,w) \\
                     &+& \sum_{s=1}^{p}t^{s}[\nu_{s}(\mu(u,v),w)+\mu(\nu_{s}(u,v),w)-\nu_{s}(u,\mu(v,w))-\mu(u,\nu_{s}(v,w))\\
                     &-& \nu_{s}(\mu(v,u),w)-\mu(\nu_{s}(v,u),w)+\nu_{s}(v,\mu(u,w))+\mu(v,\nu_{s}(u,w))]
\end{eqnarray*}
It means: $KV_{\mu_{t}}(u, v,w) = KV_{\mu}(u, v,w) +
\sum_{s=1}^{p}t^{s}d_{\mu}\nu_{s}(u, v,w)$, with:
\begin{eqnarray}
d_{\mu}\nu_{s}(u,v,w)&=&\nu_{s}(\mu(u,v),w)\nonumber\\
                     &+&\mu(\nu_{s}(u,v),w)-\nu_{s}(u,\mu(v,w))-\mu(u,\nu_{s}(v,w))-\nu_{s}(\mu(v,u),w)\nonumber\\ &-&\mu(\nu_{s}(v,u),w)+\nu_{s}(v,\mu(u,w))+\mu(v,\nu_{s}(u,w))
\end{eqnarray}
Thus, if $\mu_{t}$ is a KV-structure of algebra, we have:
\begin{eqnarray}
KV_{\mu_{t}}=0 &\Longleftrightarrow & KV_{\mu} + \sum_{s=1}^{p}t^{s}d_{\mu}\nu_{s}=0 \nonumber\\
               &\Longleftrightarrow & \sum_{s=1}^{p}t^{s} d_{\mu}\nu_{s}=0 \nonumber\\
               &\Longleftrightarrow & d_{\mu}\nu_{s}=0, s=1,\dots, p
\end{eqnarray}
 Let $(e_{1},e_{2})$ be the canonical basis of $\mathrm{I\!R}^{2}$ and $\nu_{s}: \mathrm{I\!R}^{2} \times \mathrm{I\!R}^{2}$ $\longrightarrow$ $\mathrm{I\!R}^{2}$, the algebra structures in $\mathbb{R}^{2}$ such that:
 $\nu_{s}(e_{i},e_{j})=a_{ij}^{k}e_{k}$.\\
 Thus, for all u, v , w elements of $\mathrm{I\!R}^{2}$,
\begin{eqnarray*}
 d_{\mu}\nu_{s}(u,v,w)=0 &\Longleftrightarrow & d_{\mu}\nu_{s}(e_{i},e_{j},e_{k})=0,  1 \leq i, j, k \leq 2 \nonumber
\end{eqnarray*}
This is equivalent to
 $\begin{cases}
  d_{\mu}\nu_{1}(e_{i},e_{j},e_{k})=0\\
  d_{\mu}\nu_{2}(e_{i},e_{j},e_{k})=0\\
  \ldots\\
  d_{\mu}\nu_{p}(e_{i},e_{j},e_{k})=0\\
  \end{cases}$.
 But
\begin{eqnarray*}
 d_{\mu}\nu_{1}(u,v,w)=0 &\Longleftrightarrow & d_{\mu}\nu_{1}(e_{i},e_{j},e_{k})=0,  1 \leq i, j, k \leq 2 \nonumber\\
                         &\Longleftrightarrow&\mu(\nu_{1}(e_{i},e_{j}),e_{k})+\nu_{1}(\mu(e_{i},e_{j}),e_{k})-\nu_{1}(e_{i},\mu(e_{j},e_{k})) \nonumber\\
                         &-&\mu(e_{i},\nu_{1}(e_{j},e_{k}))-\nu_{1}(\mu(e_{j},e_{i}),e_{k})-\mu(\nu_{1}(e_{j},e_{i}),e_{k}) \nonumber\\
                         &+&\nu_{1}(e_{j},\mu(e_{i},e_{k})) +\mu(e_{j},\nu_{1}(e_{i},e_{k}))=0
\end{eqnarray*}
\begin{eqnarray}
&\Longleftrightarrow&
\begin{cases}\label{eq7}
 (\Gamma_{12}^{1}-\Gamma_{21}^{1})a_{11}^{1}+\Gamma_{22}^{1}a_{11}^{2}+(\Gamma_{11}^{1}-\Gamma_{21}^{2})a_{12}^{1}+\Gamma_{21}^{1}a_{12}^{2}
 +(\Gamma_{12}^{2}-\Gamma_{11}^{1}-\Gamma_{21}^{2})a_{21}^{1}\\
 -(\Gamma_{12}^{1}+\Gamma_{21}^{1})a_{21}^{2}+\Gamma_{11}^{2}a_{22}^{1}=0\\
\Gamma_{21}^{2}a_{11}^{1}+(\Gamma_{12}^{1}-2\Gamma_{21}^{1}+\Gamma_{22}^{2})a_{11}^{2}+\Gamma_{11}^{2}a_{12}^{1}-2\Gamma_{11}^{2}a_{21}^{1}
+(\Gamma_{11}^{1}-2\Gamma_{21}^{2})a_{21}^{2}\\
+\Gamma_{11}^{2}a_{22}^{2}=0\\
-\Gamma_{22}^{1}a_{11}^{1}+(2\Gamma_{12}^{1}-\Gamma_{22}^{2})a_{12}^{1}+2\Gamma_{22}^{1}a_{12}^{2}-\Gamma_{22}^{1}a_{21}^{2}
 +(2\Gamma_{12}^{2}-\Gamma_{11}^{1}-\Gamma_{21}^{2})a_{22}^{1}-\Gamma_{12}^{1}a_{22}^{2}=0\\
-\Gamma_{22}^{1}a_{11}^{2}+(\Gamma_{12}^{2}+\Gamma_{21}^{2})a_{12}^{1}+(\Gamma_{12}^{1}-\Gamma_{21}^{1}+\Gamma_{22}^{2})a_{12}^{2}\\
-\Gamma_{12}^{1}a_{21}^{1}+(\Gamma_{12}^{1}-\Gamma_{22}^{2})a_{21}^{2}-\Gamma_{11}^{2}a_{22}^{1}+(\Gamma_{12}^{2}-\Gamma_{21}^{2})a_{22}^{2}=0\\
 \end{cases}
 \end{eqnarray}
Similarly $d_{\mu}\nu_{s}(u,v,w)=0 , 2\leq s\leq p$ equals
Eq.~(\ref{eq7}) which is the equation Eq.~(\ref{eq4}), defining
$Ker\delta^{2}$.
  Thus, for all $u$, $v$ and $w$ elements of $\mathrm{I\!R}^{2}$,
\begin{center}
  $d_{\mu}\nu_{s}(u,v,w)=0  \Longleftrightarrow  \nu_{s} \in Ker\delta^{2}$
\end{center}
\end{proof}
\subsection{The case of the KV-Hessian structure}
In this subsection, we determine the deformations of a Hessian
KV-structure on $\mathrm{I\! R}^{2}$. The following proposition
characterizes these deformations.
\begin{theorem}
Let $\mu$ be a KV-Hessian structure defined by
\begin{center}
  $\mu =\biggl(\begin{pmatrix}
          0&a\\
          a&0
            \end{pmatrix},
        \begin{pmatrix}
         b&0\\
         0&a
            \end{pmatrix}\biggr),   a\neq b \neq 0$
\end{center}
In a $(e_{1},e_{2})$ basis of $\mathrm{I\!R}^{2}$, the $\mu_{t}$
deformations of the $\mu$ KV-structure are:
\begin{eqnarray*}
        \biggl(\begin{pmatrix}
         a_{11}^{1}\sum_{i\geq 1}t^{i}  & a+ a_{12}^{1}\sum_{i\geq 1}t^{i}\\
           a+ a_{12}^{1}\sum_{i\geq 1}t^{i}&a_{22}^{1}\sum_{i\geq 1}t^{i}
            \end{pmatrix},
        \begin{pmatrix}
        b+ a_{11}^{2}\sum_{i\geq 1}t^{i}&(a_{21}^{2}+\frac{1}{2}a_{12}^{1})\sum_{i\geq 1}t^{i}\\
        (a_{21}^{2}+\frac{1}{2}a_{12}^{1})\sum_{i\geq 1}t^{i} &a+a_{12}^{1}\sum_{i\geq 1}t^{i}
            \end{pmatrix}\biggr)
\end{eqnarray*}
\end{theorem}
\begin{proof}
From equation Eq.~(\ref{eq7}), we have: $\begin{cases}
aa_{12}^{2}-2aa_{21}^{2}+ba_{22}^{1}=0\\
ba_{12}^{1}-2ba_{21}^{1}+ba_{22}^{2}=0\\
aa_{12}^{1}-aa_{22}^{2}=0\\
aa_{12}^{2}-aa_{21}^{1}-ba_{22}^{1}=0\\
     \end{cases}$. We deduce that:\\
$a_{22}^{2}=a_{21}^{1}=a_{12}^{1},
a_{12}^{2}=a_{21}^{2}+\frac{1}{2}a_{12}^{1}$. So:
     $\nu_{i} =\biggl(\begin{pmatrix}
         a_{11}^{1}&a_{12}^{1}\\
         a_{12}^{1}&a_{22}^{1}
            \end{pmatrix},
        \begin{pmatrix}
          a_{11}^{2}&a_{21}^{2}+\frac{1}{2}a_{12}^{1}\\
         a_{21}^{2}+\frac{1}{2}a_{12}^{1}&a_{12}^{1}
            \end{pmatrix}\biggr)$.\\
Hence the $\mu_{t}$ deformations of the $\mu$ Hessian KV-structure:
\begin{eqnarray*}
        \biggl(\begin{pmatrix}
         a_{11}^{1}\sum_{i\geq 1}t^{i}  & a+ a_{12}^{1}\sum_{i\geq 1}t^{i}\\
           a+ a_{12}^{1}\sum_{i\geq 1}t^{i}&a_{22}^{1}\sum_{i\geq 1}t^{i}
            \end{pmatrix},
        \begin{pmatrix}
        b+ a_{11}^{2}\sum_{i\geq 1}t^{i}&(a_{21}^{2}+\frac{1}{2}a_{12}^{1})\sum_{i\geq 1}t^{i}\\
        (a_{21}^{2}+\frac{1}{2}a_{12}^{1})\sum_{i\geq 1}t^{i} &a+a_{12}^{1}\sum_{i\geq 1}t^{i}
            \end{pmatrix}\biggr)
\end{eqnarray*}

\end{proof}

\section{General conclusion}\label{sec5}
Deformation quantization of a Hessian structure associated with a
Koszul-Vinberg algebra opens up fascinating perspectives in both
representation theory and algebraic geometry. Through our study, we
have demonstrated how Hessian properties can be integrated into the
process of deformation to obtain a rich and coherent quantum
description. We have highlighted the profound interactions between
the algebraic structure of Koszul-Vinberg algebras on
$\mathrm{I\!R}^{2}$ and the  geometric aspects of quantization. The
results obtained highlight   not only the importance of Hessian
structures in the context of    quantization, but also their
potential to shed light on open questions     in theoretical
physics, such as the formulation of quantum field     theories and the understanding of dynamical systems.\\
In addition, the concrete examples we have provided illustrate the
relevance of our approach to solving specific problems. \backmatter

\bmhead{Supplementary information} This manuscript has no additional
data.

\bmhead{Acknowledgments} We would like to thank all the active
members of the algebra and geometry research group at the University
of Maroua in Cameroon.

\section*{Declarations}
This article has no conflict of interest to the journal. No
financing with a third party.
\begin{itemize}
\item No Funding
\item No Conflict of interest/Competing interests (check journal-specific guidelines for which heading to use)
\item  Ethics approval
\item  Consent to participate
\item  Consent for publication
\item  Availability of data and materials
\item  Code availability
\item Authors' contributions
\end{itemize}
\bibliography{bibliography}

\end{document}